\documentclass[preprint,reqno]{imsart}

\RequirePackage[OT1]{fontenc}
\RequirePackage{amsthm,amsmath}
\RequirePackage[numbers]{natbib}
\RequirePackage[colorlinks,citecolor=blue,urlcolor=blue]{hyperref}

\usepackage{bbm}
\usepackage{enumitem}
\usepackage{mathtools}
\usepackage{amsmath}
\usepackage{mathrsfs} 
\usepackage[title,page]{appendix}
\usepackage[margin=1in]{geometry}
\startlocaldefs
\numberwithin{equation}{section}
\theoremstyle{plain}

\endlocaldefs
\newtheorem{theorem}{Theorem}[section]
\newtheorem{lemma}[theorem]{Lemma}
\newtheorem{remark}[theorem]{Remark}

\newtheorem{proposition}[theorem]{Proposition}

\allowdisplaybreaks[1]
\begin{document}

\begin{frontmatter}
\title{Functional approximations via Stein's method of exchangeable pairs}
\runtitle{Functional Stein's method with exchangeable pairs}

\begin{aug}
\author{\fnms{Miko{\l}aj J.} \snm{Kasprzak}\ead[label=e1]{mikolaj.kasprzak@uni.lu}},

\runauthor{Miko{\l}aj J. Kasprzak}

\affiliation{University of Oxford}

\address{University of Luxembourg\\
Department of Mathematics\\
Maison du Nombre\\
6 Avenue de la Fonte\\
L-4364 Esch-sur-Alzette\\
Luxembourg\\
\printead{e1}\\
\phantom{E-mail:\ }}
\end{aug}

\begin{abstract}
We combine the method of exchangeable pairs with Stein's method for functional approximation. As a result, we give a general linearity condition under which an abstract Gaussian approximation theorem for stochastic processes holds. We apply this approach to estimate the distance of a sum of random variables, chosen from an array according to a random permutation, from a Gaussian mixture process. This result lets us prove a functional combinatorial central limit theorem. We also consider a graph-valued process and bound the speed of convergence of the distribution of its rescaled edge counts to a continuous Gaussian process.\\~\\
\textbf{R\'esum\'e:} Nous combinons la m\'ethode des paires \'echangeables avec la m\'ethode d’approximation fonctionnelle de Stein. De cette fa{\c c}on, nous obtenons une condition g\'en\'erale de lin\'earit\'e sous laquelle un r\'esultat abstrait d’approximation Gaussienne est valide. Nous appliquons cette approche \`a l’estimation de la distance entre une somme de variables al\'eatoires, choisies dans un tableau par le biais d’une permutation al\'eatoire, et un m\'elange de processus Gaussiens. \`A partir de ce r\'esultat, nous prouvons un th\'eor\`eme central limite fonctionnel combinatoire. Nous consid\'erons \'egalement un graphe al\'eatoire et fournissons des bornes pour la vitesse de convergence de la loi de son nombre d’ar\^etes (apr\'es un changement d'\'echelle) vers un processus Gaussien continu.
\end{abstract}


\end{frontmatter}
\section{Introduction}
In \cite{stein} Stein observed that a random variable $Z$ has the standard normal law if and only if 
\begin{equation}\label{stein_identity}
\mathbbm{E}Zf(Z)=\mathbbm{E}f'(Z)
\end{equation}
for all smooth functions $f$. Therefore, if, for a random variable $W$ with mean zero and variance $1$, 
\begin{equation}\label{stein_equation}
|\mathbbm{E}f'(W)-\mathbbm{E}Wf(W)|
\end{equation}
 is close to zero for a large class of functions $f$, then the law of $W$ should be approximately Gaussian. In \cite{stein1}, Stein combined this observation with his \textit{exchangeable-pair} approach. Therein, for a centred and scaled random variable $W$, its copy $W'$ is constructed in such a way that $(W,W')$ forms an exchangeable pair and the linear regression condition:
\begin{equation}\label{regression_condition}
\mathbbm{E}\left[W'-W|W\right]=-\lambda W
\end{equation}
is satisfied for some $\lambda>0$. This, in many cases, simplifies the process of obtaining bounds on the distance of $W$ from the normal distribution.

This approach was extended in \cite{rinott} to examples in which an approximate linear regression condition holds:
\begin{equation*}
\mathbbm{E}\left[W'-W|W\right]=-\lambda W+R
\end{equation*}
for some remainder $R$. A multivariate version of the method was first described in \cite{meckes} and then in \cite{reinert_roellin}. In  \cite{reinert_roellin}, for an exchangeable pair of $d$-dimensional vectors $(W,W')$ the following condition is used:
\begin{equation}\label{lambda_matrix}
\mathbbm{E}[W'-W|W]=-\Lambda W+R
\end{equation}
for some invertible matrix $\Lambda$ and a remainder term $R$.
The approach of \cite{reinert_roellin} was further reinterpreted and combined with the approach of \cite{meckes} in \cite{meckes09}. 


On the other hand, in the seminal paper \cite{diffusion}, Barbour addressed the problem of providing bounds on the rate of convergence in functional limit results (or invariance principles as they are often called in the literature). He observed that Stein's logic of \cite{stein} may also be used in  the setup of the Functional Central Limit Theorem. He found a condition, similar to (\ref{stein_identity}), characterising the distribution of a standard real Wiener process. Combined with Taylor's theorem, it allowed Barbour to obtain a bound on the rate of convergence in the celebrated Donsker's invariance principle.

%
This paper is an attempt to combine the method of exchangeable pairs with functional approximations. We provide a novel approach to bounding distances of stochastic processes from Gaussian processes. Our approach is influenced by the setup of \cite{rinott}  and \cite{diffusion}.

\subsection{Motivation}
We are motivated by a number of (finite-dimensional) examples studied in Stein's method literature using exchangeable pairs, which could be extended to the functional setting. Functional limit results play an important role in applied fields. Researchers often choose to model discrete phenomena with continuous processes arising as scaling limits of discrete ones. The reason is that those scaling limits may be studied using stochastic analysis and are more robust to changes in local details. Questions about the rate of convergence in functional limit results are equivalent to ones about the error those researchers make. Obtaining bounds on a certain distance between the scaled discrete and the limiting continuous processes provides a way of quantifying this error.

We consider two main examples. The first one is a combinatorial functional central limit theorem. The second one considers a process representing edge counts in a graph-valued process created by unveiling subsequent vertices of a Bernoulli random graph as time progresses. 

The former is a functional version of the result proved qualitatively in \cite{chen_ho} and quantitatively in \cite{chen2015} and an extension of the main result of \cite{functional_combinatorial}. It considers an array $\left\lbrace X_{i,j}:i,j=1,\cdots,n\right\rbrace$ of i.i.d. random variables, which are then used to create a stochastic process:
\begin{equation}\label{process}
t\mapsto \frac{1}{s_n}\sum_{i=1}^{\lfloor nt\rfloor}X_{i\pi(i)},
\end{equation}
where $s_n^2$ is the variance of $\sum_{i=1}^n X_{i\pi(i)}$ and $\pi$ is a uniform random permutation on $\lbrace 1,\cdots,n\rbrace$. The motivation for studying this and similar topics comes from permutation tests in non-parametric statistics. Similar setups, yet with a deterministic array of numbers, and in a finite-dimensional context, have also been considered by other authors (see \cite{wald_wolfowitz} for one of the first works on this topic and \cite{bolthausen}, \cite{goldstein}, \cite{neammanee} for quantitative results).

The second example, which considers Bernoulli random graphs, goes back to \cite{Janson1991}. It was first studied using exchangeable pairs in a finite-dimensional context in \cite{reinert_roellin1}, where a random vector whose components represent statistics corresponding to the number of edges, two-stars and triangles is studied. The authors bound its distance from a normal distribution. We consider a functional analogue of this result, concentrating, for simplicity, only on the number of edges.
 Our approach can, however, be also extended to encompass the number of two-stars and triangles using a multivariate functional exchangeable-pair methodology which we leave for future work. All of those statistics are often of interest in applications, for example, when approximating the clustering coefficient of a network or in conditional uniform graph tests. 
\subsection{Contribution of the paper}
The main achievements of the paper are the following:
\begin{enumerate}[label=(\alph*)]
\item An abstract approximation theorem (Theorem \ref{theorem1}), providing a bound on the distance between a stochastic process $\mathbf{Y}_n$ valued in $\mathbbm{R}$ and a Gaussian mixture process. The theorem assumes that the process $\mathbf{Y}_n$ satisfies the linear regression condition 
\begin{equation*}
\mathbbm{E}\left\lbrace Df(\mathbf{Y}_n)\left[\mathbf{Y}_n'-\mathbf{Y}_n\right]|\mathbf{Y}_n\right\rbrace=-\lambda_nDf(\mathbf{Y}_n)[\mathbf{Y}_n]+R_f,
\end{equation*}
for all functions $f$ in a certain class of test functions, some $\lambda_n>0$ and some random variable $R_f=R_f(\mathbf{Y}_n)$. Crucially, the bound in Theorem \ref{theorem1} is derived with respect to a class of test functions so rich that the bound approaching zero fast enough, under certain assumptions, implies the law of $\mathbf{Y}_n$ converging weakly in the Skorokhod and uniform topologies on the Skorokhod space. The exact conditions under which this convergence occurs are stated in Proposition \ref{prop_m}. Theorem \ref{theorem1} is used in the derivation of the remaining results of this paper.
\item A novel functional combinatorial central limit theorem. In Theorem \ref{combi_pre_lim}, we establish a bound on the distance between process (\ref{process}) and a Gaussian mixture, piecewise constant process. Furthermore, a qualitative result showing convergence in distribution of process (\ref{process}) to a continuous Gaussian limiting process is provided in Theorem \ref{conv_cont_theorem}. Thus, we extend \cite{functional_combinatorial}, where similar results were proved under the assumption that all the $X_{i,j}$'s for $i,j=1,\cdots,n$ are deterministic. Our bound is also an extension of \cite{chen2015}, where a bound on the rate of weak convergence of the law of $\frac{1}{s_n}\sum_{i=1}^nX_{i\pi(i)}$ to the standard normal distribution is obtained.
\item A novel functional limit theorem for a statistic corresponding to edge counts in a Bernoulli random graph, together with a bound on the rate of convergence. We consider a Bernoulli random graph $G(n,p)$ on $n$ vertices with edge probabilities $p$. Letting $I_{i,j}$, for $i,j=1,\cdots,n$ be the indicator that edge $(i,j)$ is present in the graph, we study a scaled statistic representing the number of edges:
$$\mathbf{T}_n(t)=\frac{\lfloor nt\rfloor -2}{2n^2}\sum_{i,j=1}^{\lfloor nt\rfloor}I_{i,j},\quad t\in[0,1].$$
Theorem \ref{theorem_pre_limiting} provides a bound on the distance between the law of the process
\begin{equation}\label{graph_ex}
t\mapsto \mathbf{T}_n(t)-\mathbbm{E}\mathbf{T}_n(t),\quad t\in[0,1]
\end{equation}
and the law of a piecewise constant Gaussian process. Theorem \ref{theorem_continuous} bounds a distance between the law of (\ref{graph_ex}) and the distribution of a continuous Gaussian process. Weak convergence of the law of (\ref{graph_ex}) in the Skorokhod and uniform topologies on the Skorokhod space to that of the continuous Gaussian process follows from the bound as a corollary.
\end{enumerate}
\subsection{Stein's method of exchangeable pairs}
The idea behind the exchangeable-pair approach of \cite{stein1} was the following. In order to obtain a bound on a distance between the distribution of a centred and scaled random variable $W$ and the standard normal law, one can bound (\ref{stein_equation}) for functions $f$ coming from a suitable class. Supposing that we can construct a $W'$ such that $(W,W')$ is an exchangeable pair and (\ref{regression_condition}) is satisfied, we can write
\begin{align*}
0=&\mathbbm{E}\left[(f(W)+f(W'))(W-W')\right]\\
=&\mathbbm{E}\left[(f(W')-f(W))(W-W')\right]+2\mathbbm{E}\left[f(W)\mathbbm{E}[W-W'|W]\right]\\
=&\mathbbm{E}\left[(f(W')-f(W))(W-W')\right]+2\lambda\mathbbm{E}[Wf(W)].
\end{align*}
It follows that
$$\mathbbm{E}[Wf(W)]=\frac{1}{2\lambda}\mathbbm{E}\left[(f(W)-f(W'))(W-W')\right].$$
Therefore, using Taylor's theorem,
\begin{align*}
&\left|\mathbbm{E}[f'(W)]-\mathbbm{E}[Wf(W)]\right|\\
=&\left|\mathbbm{E}[f'(W)]+\frac{1}{2\lambda}\mathbbm{E}\left[(f(W')-f(W))(W-W')\right]\right|\\
\leq&\left|\mathbbm{E}f'(W)-\frac{1}{2\lambda}\mathbbm{E}f'(W)(W-W')^2\right|+\frac{\|f''\|_{\infty}}{2\lambda}\mathbbm{E}|W-W'|^3\\
\leq&\|f'\|_{\infty}\mathbbm{E}\left|\mathbbm{E}\left[\frac{1}{2\lambda}(W-W')^2|W\right]-1\right|+\frac{\|f''\|_{\infty}}{2\lambda}\mathbbm{E}|W-W'|^3\\
\leq&\frac{\|f'\|_{\infty}}{2\lambda}\sqrt{\text{Var}\left[\mathbbm{E}\left[(W-W')^2|W\right]\right]}+\frac{\|f''\|_{\infty}}{2\lambda}\mathbbm{E}|W-W'|^3,
\end{align*}
which provides the desired bound.

Before the publication of \cite{meckes, reinert_roellin, meckes09} the method was restricted to one-dimensional approximations. It was, however, also used in the context of non-normal approximations (e.g \cite{chatterjee, chatterjee1, roellin1}). More recently several authors have extended and applied the method. D{\"o}bler extended it to Beta distribution in \cite{dobler} and Chen and Fang used it for the combinatorial CLT \cite{chen2015}.
\subsection{Stein's method in its generality}

The aim of the general version of Stein's method is to find a bound of the quantity $|\mathbbm{E}_{\nu_n}h-\mathbbm{E}_{\mu}h|$, where $\mu$ is the target (known) distribution, $\nu_n$ is the approximating law and $h$ is chosen from a suitable class of real-valued test functions $\mathcal{H}$. The procedure can be described in terms of three steps. First, an operator $\mathcal{A}$ acting on a class of real-valued functions is sought, such that $$\left(\forall f\in\text{Domain}(\mathcal{A})\quad\mathbbm{E}_{\nu}\mathcal{A}f=0\right)\quad \Longleftrightarrow \quad\nu=\mu,$$
where $\mu$ is our target distribution. Then, for a given function $h\in\mathcal{H}$, the following Stein equation:
$$\mathcal{A}f=h-\mathbbm{E}_{\mu}h$$
is solved. Finally, using properties of the solution and various mathematical tools (among which the most popular are Taylor's expansions in the continuous case, Malliavin calculus, as described in \cite{nourdin}, and coupling methods), a bound is sought for the quantity $|\mathbbm{E}_{\nu_n}\mathcal{A}f_h|$.

\subsection{Functional Stein's method}
Approximations by laws of stochastic processes have not been covered in the Stein's method literature very widely, with the notable exceptions of \cite{diffusion, functional_combinatorial, shih, Coutin, decreusefond_higher} and recently \cite{kasprzak1, kasprzak2, decreusefond2, decreusefond_rough, campese}. In \cite{shih}, Stein's method is developed for approximations by abstract Wiener measures on a real separable Banach space. References \cite{kasprzak1, decreusefond2} establish a method for bounding the speed of weak convergence of continuous-time Markov chains satisfying certain assumptions to diffusion processes. Reference \cite{kasprzak2}, on the other hand, treats multi-dimensional processes represented by scaled sums of random variables with different dependence structures and establishes bounds on their distances from continuous Gaussian processes. The recent reference \cite{campese} introduces a Dirichlet structure and the corresponding Gamma calculus in an infinite-dimensional context and combines it with the methodology of \cite{shih} to derive bounds on distances from Gaussian random variables valued in Hilbert spaces. 

The bounds in references \cite{diffusion, functional_combinatorial, kasprzak1, kasprzak2} are obtained with respect to convergence-determining classes of test functions and weak convergence results in the Skorokhod topology follow from the bounds as corollaries. These references all use and adapt the setup of \cite{diffusion}. Therein, the author studies Donsker's theorem saying that for a sequence of  i.i.d. real random variables $(X_n)_{n=1}^{\infty}$ with mean zero and unit variance, the random process
\begin{equation}\label{int}
\mathbf{Y}_n(t)= n^{-1/2}\sum_{i=1}^{\lfloor nt\rfloor}X_i,\quad t\in[0,1]
\end{equation}
converges in distribution to the standard Brownian motion with respect to the Skorokhod topology.  Through a careful and technical application of the general Stein's method philosophy explained above, and a subsequent repetitive use of Taylor's theorem, Barbour \cite{diffusion} proved a powerful estimate on a distance between the law of $\mathbf{Y}_n$ in (\ref{int}) and the Wiener measure. Specifically, he considered test functions $g$ acting on the Skorokhod space $D \left( [0,1], \mathbbm{{R}} \right)$ of c\`adl\`ag real-valued maps on $[0,1]$, such that $g$ takes values in the reals, does not grow faster than a cubic, is twice Fr\'echet differentiable and its second derivative is Lipschitz. Denoting by $\mathbf{Z}$ Brownian motion on $[0,1]$ and adopting the notation of (\ref{int}), his result says that
$$|\mathbbm{E}g(\mathbf{Y}_n)-\mathbbm{E}g(\mathbf{Z})|\leq C_g\frac{\mathbbm{E}|X_1|^3+\sqrt{\log n}}{\sqrt{n}},$$
where $C_g$ is a constant, independent of $n$, yet depending on the (carefully defined) \textit{smoothness properties} of $g$. Among the applications and extensions considered by Barbour are an analysis of the empirical distribution function of i.i.d. random variables and the Wald-Wolfowitz theorem often used to construct tests in non-parametric statistics \cite{wald1940}.

The results of \cite{Coutin, decreusefond_higher, decreusefond2, decreusefond_rough} consider a totally different setup. Therein, the authors develop a Stein theory on a Hilbert space using a Besov-type topology. Their bounds do not imply weak convergence in the Skorokhod topology and, for most of the natural examples one can consider, the continuous mapping theorem does not apply in the setting they work with. For instance, as opposed to the results of \cite{diffusion}, those of \cite{Coutin} do not allow one to draw conclusions about convergence of the supremum of a process. For these reasons, the setup we use to obtain the bounds in the current paper is analogous to that of the former class of references \cite{diffusion, functional_combinatorial, kasprzak1, kasprzak2}. We consider it more flexible than the one of the class of references \cite{Coutin, decreusefond_higher, decreusefond2, decreusefond_rough}. It is also more suited for applications to processes belonging to the widely-used Skorokhod space than the one of \cite{shih}, which only caters for separable Banach spaces.

\subsection{Structure of the paper}
The paper is organised as follows. In Section \ref{section22} we introduce the spaces of test functions which will be used in the main results. We also quote a proposition showing that, under certain assumptions, they determine convergence in distribution under the uniform topology. In Section \ref{section33} we set up the Stein equation for approximation by a pre-limiting process and provide properties of the solutions. In Section \ref{section4} we provide an exchangeable-pair condition and prove an abstract exchangeable-pair-type approximation theorem. Section \ref{section5} is devoted to the functional combinatorial central limit theorem example and Section \ref{section6} discusses the graph-valued process example.
\section{Spaces $M$, $M^1$, $M^2$, $M^0$}\label{section22}
The following notation is used throughout the paper. For a function $w$ defined on the interval $[0,1]$ and taking values in a Euclidean space, we define $$\|w\|=\sup_{t\in[0,1]}|w(t)|.$$ 
We also let $D=D([0,1],\mathbbm{R})$ be the Skorokhod space of all c\`adl\`ag functions on $[0,1]$ taking values in $\mathbbm{R}$.  We will often write $\mathbbm{E}^W[\,\cdot\,]$ instead of $\mathbbm{E}[\,\cdot\,|W]$.

Let us define:
$$\|f\|_L:=\sup_{w\in D}\frac{|f(w)|}{1+\|w\|^3}\text{,}$$
and let $L$ be the Banach space of continuous functions $f:D\to\mathbbm{R}$ such that $\|f\|_L<\infty$. Following \cite{diffusion}, we now let $M\subset L$ consist of the twice Fr\'echet differentiable functions $f$, such that:
\begin{equation}\label{space_m}
\|D^2f(w+h)-D^2f(w)\|\leq k_f\|h\|\text{,}
\end{equation}
for some constant $k_f$, uniformly in $w,h\in D$. By $D^kf$ we mean the $k$-th Fr\'echet derivative of $f$ and the norm of a $k$-linear form $B$ on $L$ is defined to be $\|B\|=\sup_{\lbrace h:\|h\|=1\rbrace} |B[h,...,h]|$. Note the following lemma, which can be proved in an analogous way to that used to show (2.6) and (2.7) of \cite{diffusion}. We omit the proof here. 
\begin{lemma}\label{first_der}
For every $f\in M$, let:
\begin{align*}
\|f\|_M:=&\sup_{w\in D}\frac{|f(w)|}{1+\|w\|^3}+\sup_{w\in D}\frac{\|Df(w)\|}{1+\|w\|^2}+\sup_{w\in D}\frac{\|D^2f(w)\|}{1+\|w\|}+\sup_{w,h\in D}\frac{\|D^2f(w+h)-D^2f(w)\|}{\|h\|}.
\end{align*}
Then, for all $f\in M$, we have $\|f\|_M<\infty$.
\end{lemma}
For future reference, we let $ M^1\subset M$ be the class of functionals $g\in M$ such that:
\begin{align}
\|g\|_{M^1}:=&\sup_{w\in D}\frac{|g(w)|}{1+\|w\|^3}+\sup_{w\in D}\|Dg(w)\|+\sup_{w\in D}\|D^2g(w)\|+\sup_{w,h\in D}\frac{\|D^2g(w+h)-D^2g(w)\|}{\|h\|}<\infty\label{m_1}
\end{align}
and $ M^2\subset M$ be the class of functionals $g\in M$ such that:
\begin{align}
\|g\|_{M^2}:=&\sup_{w\in D}\frac{|g(w)|}{1+\|w\|^3}+\sup_{w\in D}\frac{\|Dg(w)\|}{1+\|w\|}+\sup_{w\in D}\frac{\|D^2g(w)\|}{1+\|w\|}+\sup_{w,h\in D}\frac{\|D^2g(w+h)-D^2g(w)\|}{\|h\|}<\infty\text{.}\label{m_2}
\end{align}

We also let $M^0$ be the class of functionals $g\in M$ such that:
\begin{align}
\|g\|_{M^0}:=&\sup_{w\in D}|g(w)|+\sup_{w\in D}\|Dg(w)\|+\sup_{w\in D}\|D^2g(w)\|+\sup_{w,h\in D}\frac{\|D^2g(w+h)-D^2g(w)\|}{\|h\|}<\infty\text{.}\nonumber
\end{align}
We note that $M^0\subset M^1\subset M^2\subset M$. 
The next proposition is \cite[Proposition 3.1]{functional_combinatorial} and shows conditions, under which convergence of the sequence of expectations of a functional $g$ under the approximating measures to the expectation of $g$ under the target measure for all $g\in M^0$ implies weak convergence of the measures of interest.
%
\begin{proposition}[Proposition 3.1 of \cite{functional_combinatorial}]\label{prop_m}
Suppose that, for each $n\geq 1$, the random element $\mathbf{Y}_n$ of $D$ is piecewise constant with intervals of constancy of length at least $r_n$. Let $\left(\mathbf{Z}_n\right)_{n\geq 1}$ be random elements of $D$ converging in distribution in $D$, with respect to the Skorokhod topology, to a random element $\mathbf{\mathbf{Z}}\in C\left([0,1],\mathbbm{R}\right)$. If:
\begin{equation}\label{assumption}
|\mathbbm{E}g(\mathbf{Y}_n)-\mathbbm{E}g(\mathbf{\mathbf{Z}}_n)|\leq C\mathscr{T}_n\|g\|_{M^0}
\end{equation}
for each $g\in M^0$ and if $\mathscr{T}_n\log^2(1/r_n)\xrightarrow{n\to\infty}0$, then the law of $\mathbf{Y}_n$ converges weakly to that of  $\mathbf{\mathbf{Z}}$ in $D$, in both the uniform and the Skorokhod topologies.
\end{proposition}
\section{Setting up Stein's method for the pre-limiting approximation}\label{section33}
The steps of the construction presented in this section will be similar to those used to set up Stein's method in \cite{diffusion} and \cite{kasprzak2}. After defining the process $\mathbf{D}_n$ whose distribution will be the target measure in Stein's method, we will construct a process $\left(\mathbf{W}_n(\cdot,u):u\geq 0\right)$ for which the target measure is stationary. We will then calculate its infinitesimal generator $\mathcal{A}_n$ and take it as our Stein operator. Next, we solve the Stein equation $\mathcal{A}_nf=g$ using the analysis of \cite{kasprzak} and prove some properties of the solution $f_n=\phi_n(g)$, with the most important one being that its second Fr\'echet derivative is Lipschitz.
\subsection{Target measure}
Let
\begin{equation}\label{d_n}
\mathbf{D}_n(t)=\sum_{i_1,\cdots,i_m=1}^{n}\tilde{Z}_{i_1,\cdots,i_m}J_{i_1,\cdots,i_m}(t),\quad t\in[0,1],
\end{equation}
where $\tilde{Z}_{i_1,\cdots,i_m}$'s are centred Gaussian and:
\begin{enumerate}[label=\Alph*)]
\item the covariance matrix $\Sigma_n\in\mathbbm{R}^{(n^m)\times(n^m)}$ of the vector $\tilde{Z}$ is positive definite, where $\tilde{Z}\in \mathbbm{R}^{(n^m)}$ is formed out of the $\tilde{Z}_{i_1,\cdots,i_m}$'s in such a way that they appear in the lexicographic order.
\item The collection $\left\lbrace J_{i_1,\cdots,i_m}\in D\left([0,1],\mathbbm{R}\right):i_1,\cdots,i_m\in\lbrace 1,\cdots, n\rbrace\right\rbrace$ is independent of the collection \\$\left\lbrace \tilde{Z}_{i_1,\cdots,i_m}:i_1,\cdots,i_m\in\lbrace 1,\cdots, n\rbrace\right\rbrace$. A typical example would be $J_{i_1,\cdots,i_m}=\mathbbm{1}_{A_{i_1,\cdots,i_m}}$ for some measurable set $A_{i_1,\cdots,i_m}$.
\end{enumerate}
\begin{remark}
It is worth noting that processes $\mathbf{D}_n$ taking the form (\ref{d_n}) often approximate interesting continuous Gaussian processes very well. An example is a Gaussian scaled random walk, i.e. $\mathbf{D}_n$ of (\ref{d_n}), where  all the $\tilde{Z}_{i_1,\cdots,i_m}$'s are standard normal and independent, $m=1$ and $J_i=\mathbbm{1}_{[i/n,1]}$ for all $i=1,\cdots,n$. It approximates Brownian Motion. By Proposition \ref{prop_m}, under several assumptions, proving by Stein's method that a piece-wise constant process $\mathbf{Y}_n$ is close enough to process $\mathbf{D}_n$ proves $\mathbf{Y}_n$'s convergence in law to the continuous process that $\mathbf{D}_n$ approximates.
\end{remark}

Now let $\lbrace (\mathscr{X}_{i_1,\cdots,i_m}(u),u\geq 0):i_1,\cdots,i_m=1,...,n\rbrace$ be an array of i.i.d. Ornstein-Uhlenbeck processes with stationary law $\mathcal{N}(0,1)$, independent of the $J_{i_1,\cdots,i_m}$'s. Consider $\tilde{\mathscr{U}}(u)=\left(\Sigma_n\right)^{1/2}\mathscr{X}(u)$, where $\mathscr{X}(u)\in\mathbbm{R}^{n^m}$ is formed out of the $\mathscr{X}_{i_1,\cdots,i_m}(u)$'s in such a way that they appear in the same order as the $\tilde{Z}_{i_1,\cdots,i_m}$'s appear in $\tilde{Z}$.
Write $\mathscr{U}_{i_1,\cdots,i_m}(u)=\left(\tilde{\mathscr{U}}(u)\right)_{I(i_1,\cdots,i_m)}$ using the bijection $I:\lbrace (i_1,\cdots,i_m):i_1,\cdots,i_m=1,\cdots, n\rbrace\to\lbrace 1,\cdots,n^m\rbrace$, given by:
\begin{equation}\label{i}
I(i_1,\cdots,i_m)=(i_1-1)n^{m-1}+\cdots+(i_{m-1}-1)n+i_m.
\end{equation}
Consider a process:
$$\mathbf{W}_n(t,u)=\sum_{i_1,\cdots,i_m=1}^{n}\mathscr{U}_{i_1,\cdots,i_m}(u)J_{i_1,\cdots,i_m}(t),\quad t\in[0,1],u\geq 0.$$
It is easy to see that the stationary law of the process $\left(\mathbf{W}_n(\cdot,u)\right)_{u\geq 0}$ is exactly the law of $\mathbf{D}_n$. 
\subsection{Stein equation}
By \cite[Propositions 4.1 and 4.4]{kasprzak2}, the following result is immediate:
\begin{proposition}\label{prop12.7}
The infinitesimal generator of the process $\left(\mathbf{W}_n(\cdot,u)\right)_{u\geq 0}$ acts on any $f\in M$ in the following way:
\begin{align*}
&\mathcal{A}_nf(w)=-Df(w)[w]+\mathbbm{E}D^2f(w)\left[\mathbf{D}_n,\mathbf{D}_n\right].
\end{align*}
Moreover, for any $g\in M$ such that $\mathbbm{E}g(\mathbf{D}_n)=0$, the Stein equation $\mathcal{A}_nf_n=g$ is solved by:
\begin{equation}\label{phi}
f_n=\phi_n(g)=-\int_0^{\infty}T_{n,u}gdu,
\end{equation}
where $(T_{n,u}f)(w)=\mathbbm{E}\left[f(we^{-u}+\sqrt{1-e^{-2u}}\mathbf{D}_n(\cdot)\right]$ Furthermore, for $g\in M$:
\begin{align}
\text{A)} \quad &\|D\phi_n(g)(w)\|\leq \|g\|_{ M}\left(1+\frac{2}{3}\|w\|^2+\frac{4}{3}\mathbbm{E}\|\mathbf{D}_n\|^2\right)\text{,}\nonumber\\
\text{B)} \quad &\|D^2\phi_n(g)(w)\|\leq \|g\|_{ M}\left(\frac{1}{2}+\frac{\|w\|}{3}+\frac{\mathbbm{E}\|\mathbf{D}_n\|}{3}\right)\text{,}\nonumber\\
\text{C)}\quad&\frac{\left\|D^2\phi_n(g)(w+h)-D^2\phi_n(g)(w)\right\|}{\|h\|}
\leq\sup_{w,h\in D}\frac{\|D^2(g+c)(w+h)-D^2(g+c)(w)\|}{3\|h\|},\nonumber
\end{align}
for any constant function $c:D\to\mathbbm{R}$ and for all $w,h\in D$.
Moreover, for all $g\in M^1$, as defined in (\ref{m_1}),
\begin{align}
\text{A)}\quad&\|D\phi_n(g)(w)\|\leq \|g\|_{ M^1},\nonumber\\
\text{B)}\quad&\|D^2\phi_n(g)(w)\|\leq \frac{1}{2}\|g\|_{ M^1}\nonumber
\end{align}
and for all $g\in M^2$, as defined in (\ref{m_2}),
\begin{align}
\|D\phi_n(g)(w)\|\leq \|g\|_{ M^2}\nonumber.
\end{align}
\end{proposition}
\section{An abstract approximation theorem}\label{section4}
We now present a theorem which provides an expression for a bound on the distance between some process $\mathbf{Y}_n$ and $\mathbf{D}_n$, defined by (\ref{d_n}), provided that we can find some $\mathbf{Y}_n'$ such that $(\mathbf{Y}_n,\mathbf{Y}_n')$ is an exchangeable pair satisfying an appropriate condition. 
\begin{theorem}\label{theorem1}
Assume that $(\mathbf{Y}_n,\mathbf{Y}_n')$ is an exchangeable pair of $D\left([0,1],\mathbbm{R}\right)$-valued random variables such that:
\begin{equation}
\mathbbm{E}^{\mathbf{Y}_n}Df(\mathbf{Y}_n)\left[\mathbf{Y}_n'-\mathbf{Y}_n\right]=-\lambda_nDf(\mathbf{Y}_n)[\mathbf{Y}_n]+R_f,
\label{condition}
\end{equation}
where $\mathbbm{E}^{\mathbf{Y}_n}[\cdot]:=\mathbbm{E}\left[\cdot|\mathbf{Y}_n\right]$, for all $f\in M$, some $\lambda_n> 0$ and some random variable $R_f=R_f(\mathbf{Y}_n)$.  Let $\mathbf{D}_n$ be defined by (\ref{d_n}). Then, for any $g\in M$:
\begin{align*}
\left|\mathbbm{E}g(\mathbf{Y}_n)-\mathbbm{E}g(\mathbf{D}_n)\right|\leq \epsilon_1+\epsilon_2+\epsilon_3
\end{align*}
where
\begin{align*}
\epsilon_1&=\frac{\|g\|_M}{12\lambda_n}\mathbbm{E}\left[\|\mathbf{Y}_n-\mathbf{Y}_n'\|^3\right],\\
\epsilon_2&=\left|\frac{1}{2\lambda_n}\mathbbm{E}D^2f(\mathbf{Y}_n)\left[\mathbf{Y}_n-\mathbf{Y}_n',\mathbf{Y}_n-\mathbf{Y}_n'\right]-\mathbbm{E}D^2f(\mathbf{Y}_n)\left[\mathbf{D}_n,\mathbf{D}_n\right]\right|,\\
\epsilon_3&=\frac{1}{\lambda_n}|\mathbbm{E}R_f|,
\end{align*}
and $f=\phi_n(g)$, as defined by (\ref{phi}).
\end{theorem}
\begin{remark}[Relevance of terms in the bound]
Term $\epsilon_1$ measures how close $\mathbf{Y}_n$ and $\mathbf{Y}_n'$ are and how \textit{large} $\lambda_n$ is. Term $\epsilon_2$ corresponds to the comparison of the covariance structure of $\mathbf{Y}_n-\mathbf{Y}_n'$ and $\mathbf{D}_n$. Estimating this term usually requires some effort yet is possible in several applications (see Theorem \ref{combi_pre_lim} and \ref{theorem_pre_limiting} below). Term $\epsilon_3$ measures the error in the exchangeable-pair linear regression condition (\ref{condition}).
\end{remark}
\begin{remark}
The term 
$$\left|\frac{1}{2\lambda_n}\mathbbm{E}D^2f(\mathbf{Y}_n)\left[\mathbf{Y}_n-\mathbf{Y}_n',\mathbf{Y}_n-\mathbf{Y}_n'\right]-\mathbbm{E}D^2f(\mathbf{Y}_n)\left[\mathbf{D}_n,\mathbf{D}_n\right]\right|$$ in the bound obtained in Theorem \ref{theorem1} is an analogue of the second condition in \cite[Theorem 3]{meckes09}. Therein, a bound on approximation by $\mathcal{N}(0,\Sigma)$ of a $d$-dimensional vector $X$ is obtained by constructing an exchangeable pair $(X,X')$ satisfying:
$$\mathbbm{E}^X[X'-X]=\Lambda X+E\quad\text{and}\quad\mathbbm{E}^X[(X'-X)(X'-X)^T]=2\Lambda\Sigma+E'$$
for some invertible matrix $\Lambda$ and some remainder terms $E$ and $E'$. In the same spirit, Theorem \ref{theorem1} could be rewritten to assume (\ref{condition}) and:
$$\mathbbm{E}^{\mathbf{Y}_n}D^2f(\mathbf{Y}_n)\left[\mathbf{Y}_n'-\mathbf{Y}_n,\mathbf{Y}_n'-\mathbf{Y}_n\right]=2\lambda_nD^2f(\mathbf{Y}_n)\left[\mathbf{D}_n,\mathbf{D}_n\right]+R^1_f.$$
The bound would then take the form:
\begin{align*}
\left|\mathbbm{E}g(\mathbf{Y}_n)-\mathbbm{E}g(\mathbf{D}_n)\right|\leq&\frac{\|g\|_M}{12\lambda_n}\mathbbm{E}\left[\|(\mathbf{Y}_n-\mathbf{Y}_n')\|^3\right]+\frac{1}{\lambda_n}|\mathbbm{E}R_f|+\frac{1}{2\lambda_n}|\mathbbm{E}R^1_f|.
\end{align*}
\end{remark}
\begin{remark}\label{uni_multi}
While Theorem \ref{theorem1} is formulated for univariate stochastic processes belonging to the Skorokhod space $D([0,1],\mathbbm{R})$, its proof below is valid for processes belonging to $D([0,1],\mathbbm{R}^p)$, for any $p\geq 1$. 

We only present the univariate case here since in the modern approach to multivariate Stein's method of exchangeable pairs, presented in \cite{reinert_roellin, reinert_roellin1, meckes09}, the exchangeable-pair condition for multivariate approximations is written not in terms of a real $\lambda_n$, as in (\ref{condition}), but in terms of a matrix-valued coefficient $\Lambda$, as in (\ref{lambda_matrix}).

While it is possible to rewrite condition (\ref{condition}) appropriately, in terms of a matrix $\Lambda_n$, prove a corresponding bound and use it in interesting applications, we think of it as a separate problem and leave it for future research.
\end{remark}
\begin{proof}[Proof of Theorem \ref{theorem1}]
Our aim is to bound $\left|\mathbbm{E}g(\mathbf{Y}_n)-\mathbbm{E}g(\mathbf{D}_n)\right|$ by bounding $\left|\mathbbm{E}\mathcal{A}_nf(\mathbf{Y}_n)\right|$, where $f$ is the solution to the Stein equation:
$$\mathcal{A}_nf=g-\mathbbm{E}g(\mathbf{D}_n),$$
for $\mathcal{A}_n$ defined in Proposition \ref{prop12.7}. Note that, by exchangeability of $(\mathbf{Y}_n,\mathbf{Y}_n')$ and (\ref{condition}):
\begin{align*}
0=&\mathbbm{E}\left(Df(\mathbf{Y}_n')+Df(\mathbf{Y}_n)\right)\left[\mathbf{Y}_n-\mathbf{Y}_n'\right]\\
=&\mathbbm{E}\left(Df(\mathbf{Y}_n')-Df(\mathbf{Y}_n)\right)\left[\mathbf{Y}_n-\mathbf{Y}_n'\right]+2\mathbbm{E}\left\lbrace\mathbbm{E}^{\mathbf{Y}_n}Df(\mathbf{Y}_n)\left[\mathbf{Y}_n-\mathbf{Y}_n'\right]\right\rbrace\\
=&\mathbbm{E}\left(Df(\mathbf{Y}_n')-Df(\mathbf{Y}_n)\right)\left[\mathbf{Y}_n-\mathbf{Y}_n'\right]+2\lambda_n\mathbbm{E}Df(\mathbf{Y}_n)[\mathbf{Y}_n]-2\mathbbm{E}R_f
\end{align*}
and so:
\begin{equation*}
\mathbbm{E}Df(\mathbf{Y}_n)[\mathbf{Y}_n]=\frac{1}{2\lambda_n}\left\lbrace\mathbbm{E}\left(Df(\mathbf{Y}_n)-Df(\mathbf{Y}_n')\right)\left[\mathbf{Y}_n-\mathbf{Y}_n'\right]+2\mathbbm{E}R_f\right\rbrace.
\end{equation*}
Therefore:
\begin{align}
\left|\mathbbm{E}\mathcal{A}_nf(\mathbf{Y}_n)\right|
=&\left|\mathbbm{E}Df(\mathbf{Y}_n)[\mathbf{Y}_n]-\mathbbm{E}D^2f(\mathbf{Y}_n)\left[\mathbf{D}_n,\mathbf{D}_n\right]\right|\nonumber\\
=&\left|\frac{1}{2\lambda_n}\mathbbm{E}\left(Df(\mathbf{Y}_n)-Df(\mathbf{Y}_n')\right)\left[\mathbf{Y}_n-\mathbf{Y}_n'\right]-\mathbbm{E}D^2f(\mathbf{Y}_n)\left[\mathbf{D}_n,\mathbf{D}_n\right]+\frac{1}{\lambda_n}\mathbbm{E}R_f\right|\nonumber\\
\leq& \left|\frac{1}{2\lambda_n}\mathbbm{E}\left(Df(\mathbf{Y}_n)-Df(\mathbf{Y}_n')\right)\left[\mathbf{Y}_n-\mathbf{Y}_n'\right]-\frac{1}{2\lambda_n}\mathbbm{E}D^2f(\mathbf{Y}_n')\left[\mathbf{Y}_n-\mathbf{Y}_n',\mathbf{Y}_n-\mathbf{Y}_n'\right]\right|\nonumber\\
&+\left|\frac{1}{2\lambda_n}\mathbbm{E}D^2f(\mathbf{Y}_n)\left[\mathbf{Y}_n-\mathbf{Y}_n',\mathbf{Y}_n-\mathbf{Y}_n'\right]-\mathbbm{E}D^2f(\mathbf{Y}_n)\left[\mathbf{D}_n,\mathbf{D}_n\right]\right|+\frac{1}{\lambda_n}|\mathbbm{E}R_f|\nonumber\\
\leq&\frac{\|g\|_M}{12\lambda_n}\mathbbm{E}\left[\|\mathbf{Y}_n-\mathbf{Y}_n'\|^3\right]+\frac{1}{\lambda_n}|\mathbbm{E}R_f|\nonumber\\
&+\left|\frac{1}{2\lambda_n}\mathbbm{E}D^2f(\mathbf{Y}_n)\left[\mathbf{Y}_n-\mathbf{Y}_n',\mathbf{Y}_n-\mathbf{Y}_n'\right]-\mathbbm{E}D^2f(\mathbf{Y}_n)\left[\mathbf{D}_n,\mathbf{D}_n\right]\right|,\nonumber
\end{align}
where the last inequality follows by Taylor's theorem and Proposition \ref{prop12.7}.
\end{proof}
\section{A Functional Combinatorial Central Limit Theorem}\label{section5}
In this section we consider a functional version of the result proved in \cite{chen_ho}. Our object of interest is a stochastic process represented by a scaled sum of independent random variables chosen from an $n\times n$ array. Only one random variable is picked from each row and for row $i$, the corresponding random variable is picked from column $\pi(i)$, where $\pi$ is a random permutation on $[n]=\lbrace 1,\cdots,n\rbrace$. Theorem \ref{combi_pre_lim} establishes a bound on the distance between this process and a pre-limiting process and Theorem \ref{conv_cont_theorem} shows convergence of this process, under certain assumptions, to a continuous Gaussian process.

Our analysis in this section is similar to that of \cite{functional_combinatorial}, where the summands in the scaled sums are chosen from a deterministic array. The authors therein also establish bounds on the approximation by a pre-limit Gaussian process and show convergence to a continuous Gaussian process. Furthermore, they establish a bound on the distance from the continuous Gaussian process for a restricted class of test functions. For random arrays the situation is more involved.

Our setup is analogous to the one considered in \cite{chen2015}, where a bound on the rate of convergence in the one-dimensional combinatorial central limit theorem is obtained using Stein's method of exchangeable pairs.
\subsection{Introduction}\label{intro_section}
Let $\mathbbm{X}=\lbrace X_{ij}:i,j\in[n]\rbrace$, where $[n]=\lbrace1,2,\cdots,n\rbrace$, be an $n\times n$ array of independent $\mathbbm{R}$-valued random variables, where $n\geq 2$, $\mathbbm{E}X_{ij}=c_{ij}$, $\text{Var} X_{ij}=\sigma_{ij}^2\geq 0$ and $\mathbbm{E}|X_{ij}|^3<\infty$. Suppose that $c_{i\cdot}=c_{\cdot j}=0$ where $c_{i\cdot}=\sum_{j=1}^n\frac{c_{ij}}{n}=\mathbbm{E}X_{i\pi(i)}$, $c_{\cdot j}=\sum_{i=1}^n\frac{c_{ij}}{n}$. Let $\pi$ be a uniform random permutation of $[n]$, independent of $\mathbbm{X}$ and for
\begin{equation}\label{s_n}
s_n^2=\frac{1}{n}\sum_{i,j=1}^n \sigma_{ij}^2+\frac{1}{n-1}\sum_{i,j=1}^n c_{ij}^2.
\end{equation}
let
$$\mathbf{Y}_n(t)=\frac{1}{s_n}\sum_{i=1}^{\lfloor nt\rfloor}X_{i\pi(i)}=\frac{1}{s_n}\sum_{i=1}^nX_{i\pi(i)}\mathbbm{1}_{[i/n,1]}(t),\quad t\in[0,1].$$

We note that $s_n^2=\text{Var}\left[\sum_{i=1}^nX_{i\pi(i)}\right]$, by the first part of \cite[Theorem 1.1]{chen2015}. 
The process $\mathbf{Y}_n$ is similar to the process $Y$ considered in \cite{functional_combinatorial} and defined by (1.4) therein with the most important difference being that we allow the $X_{i,j}$'s to be random, whereas the authors in \cite{functional_combinatorial} assumed them to be deterministic. Bounds on the distance between one-dimensional distributions of $\mathbf{Y}_n$ and a normal distribution have been obtained via Stein's method in \cite[Theorem 1.1]{chen2015}.
\subsection{Exchangeable pair setup}\label{ex_section}
Select uniformly at random two different indices $I,J\in [n]$ and let:
$$\mathbf{Y}_n'=\mathbf{Y}_n-\frac{1}{s_n}X_{I\pi(I)}\mathbbm{1}_{[I/n,1]}-\frac{1}{s_n}X_{J\pi(J)}\mathbbm{1}_{[J/n,1]}+\frac{1}{s_n}X_{I\pi(J)}\mathbbm{1}_{[I/n,1]}+\frac{1}{s_n}X_{J\pi(I)}\mathbbm{1}_{[J/n,1]}.$$
Note that $(\mathbf{Y}_n,\mathbf{Y}_n')$ is an exchangeable pair and that for all $f\in M$:
\begin{align*}
&\mathbbm{E}^{\mathbf{Y}_n}\left\lbrace Df(\mathbf{Y}_n)\left[\mathbf{Y}_n-\mathbf{Y}_n'\right]\right\rbrace\\
=&\frac{1}{s_n}\mathbbm{E}^{\mathbf{Y}_n}\left\lbrace Df(\mathbf{Y}_n)\left[X_{I\pi(I)}\mathbbm{1}_{[I/n,1]}+X_{J\pi(J)}\mathbbm{1}_{[J/n,1]}-X_{I\pi(J)}\mathbbm{1}_{[I/n,1]}-X_{J\pi(I)}\mathbbm{1}_{[J/n,1]}\right]\right\rbrace\\
=&\frac{1}{n(n-1)s_n}\sum_{i, j=1}^n\mathbbm{E}^{\mathbf{Y}_n}\left\lbrace Df(\mathbf{Y}_n)\left[X_{i\pi(i)}\mathbbm{1}_{[i/n,1]}+X_{j\pi(j)}\mathbbm{1}_{[j/n,1]}-X_{i\pi(j)}\mathbbm{1}_{[i/n,1]}-X_{j\pi(i)}\mathbbm{1}_{[j/n,1]}\right]\right\rbrace\\
=&\frac{2}{n-1}Df(\mathbf{Y}_n)[\mathbf{Y}_n]-\frac{2}{n(n-1)s_n}\sum_{i,j=1}^n\mathbbm{E}^{\mathbf{Y}_n}Df(\mathbf{Y}_n)\left[X_{i,\pi(j)}\mathbbm{1}_{[i/n,1]}\right]. 
\end{align*}

Therefore:
\begin{align}\label{ex_condition}
\mathbbm{E}^{\mathbf{Y}_n}\left\lbrace Df(\mathbf{Y}_n)[\mathbf{Y}_n'-\mathbf{Y}_n]\right\rbrace=&-\frac{2}{n-1}Df(\mathbf{Y}_n)\left[\mathbf{Y}_n\right]+\frac{2}{n(n-1)s_n}\sum_{i,j=1}^nDf(\mathbf{Y}_n)\left[\mathbbm{E}^{\mathbf{Y}_n}[X_{i,\pi(j)}]\mathbbm{1}_{[i/n,1]}\right].
\end{align}
So condition (\ref{condition}) is satisfied with
$$\lambda_n=\frac{2}{n-1}\quad\text{and}\quad R_f=\frac{2}{n(n-1)s_n}\sum_{i,j=1}^nDf(\mathbf{Y}_n)\left[\mathbbm{E}^{\mathbf{Y}_n}[X_{i,\pi(j)}]\mathbbm{1}_{[i/n,1]}\right].$$

\subsection{Pre-limiting process}\label{section_pre_lim}
Now let $\hat{Z}_i=\frac{1}{\sqrt{n-1}}\sum_{l=1}^nX_{il}''\left(Z_{il}-\frac{1}{n}\sum_{j=1}^nZ_{jl}\right)$, for $\mathbbm{X}''=\lbrace X_{ij}'':i,j\in[n]\rbrace$ being an independent copy of $\mathbbm{X}$ and $Z_{il}$'s i.i.d. standard normal, independent of  $\mathbbm{X}\cup\mathbbm{X}'.$ Then, let
\begin{equation}\label{pre_lim_an}
\mathbf{D}_n(t)=\frac{1}{s_n}\sum_{i=1}^{\lfloor nt\rfloor}\hat{Z}_i,\quad t\in[0,1].
\end{equation}
We will compare the distribution of $\mathbf{Y}_n$ with the distribution of $\mathbf{D}_n$. $\mathbf{D}_n$ is a conceptually easy process with the same covariance structure as $\mathbf{Y}_n$. It is constructed in a way similar to the process in \cite[(3.13)]{functional_combinatorial}.
Note that $\hat{Z}_i$ has mean zero for all $i$ and
\begin{align}
\mathbbm{E}\hat{Z}_i^2=&\frac{1}{n-1}\sum_{l=1}^n\mathbbm{E}\left[X_{il}^2\right]\mathbbm{E}\left[\left(Z_{il}-\frac{1}{n}\sum_{j=1}^nZ_{jl}\right)^2\right]\nonumber\\
&+\frac{1}{n-1}\sum_{1\leq l\neq k\leq n}\mathbbm{E}\left[X_{il}X_{ik}\right]\mathbbm{E}\left[\left(Z_{il}-\frac{1}{n}\sum_{j=1}^nZ_{jl}\right)\left(Z_{ik}-\frac{1}{n}\sum_{j=1}^nZ_{jk}\right)\right]\nonumber\\
=&\frac{1}{n-1}\sum_{l=1}^n\mathbbm{E}\left[X_{il}^2\right]\left(1-\frac{2}{n}+\frac{1}{n}\right)\nonumber\\
=&\frac{1}{n}\sum_{l=1}^n\mathbbm{E}X_{il}^2\nonumber\\
=&\frac{1}{2n^2}\left(2(n-1)\sum_{l=1}^n\mathbbm{E}X_{il}^2+2\sum_{r=1}^n\mathbbm{E}X_{ir}^2\right)\nonumber\\
=&\frac{1}{2n^2}\left(\sum_{1\leq k\neq l\leq n} \mathbbm{E}\left[\left(X_{ik}-X_{il}\right)^2\right]+2\sum_{1\leq k\neq l\leq n}\mathbbm{E}X_{ik}\mathbbm{E}X_{il}+2\sum_{r=1}^n\mathbbm{E}X_{ir}^2\right)\nonumber\\
=&\frac{1}{2n^2}\left(\sum_{1\leq k\neq l\leq n} \mathbbm{E}\left[\left(X_{ik}-X_{il}\right)^2\right]+2\sum_{r=1}^n\sigma_{ir}^2\right)\label{e_zi}
\end{align}
as $c_{i\cdot}=0$, and, for $i\neq j$,
\begin{align}
\mathbbm{E}\hat{Z}_i\hat{Z}_j=&\frac{1}{n-1}\sum_{k,l=1}^n\mathbbm{E}(X_{ik}X_{jl})\mathbbm{E}\left[\left(Z_{ik}-\frac{1}{n}\sum_{r=1}^nZ_{rk}\right)\left(Z_{jl}-\frac{1}{n}\sum_{r=1}^nZ_{rl}\right)\right]\nonumber\\
=&-\frac{1}{n(n-1)}\sum_{k=1}^nc_{ik}c_{jk}\nonumber\\
=&\frac{1}{2n^2(n-1)}\left(2\sum_{k=1}^n\left(-\mathbbm{E}X_{ik}\right)\mathbbm{E}X_{jk}-2(n-1)\sum_{k=1}^n\mathbbm{E}X_{ik}\mathbbm{E}X_{jk}\right)\nonumber\\
=&\frac{1}{2n^2(n-1)}\left(2\sum_{1\leq k\neq l\leq n}\mathbbm{E}X_{il}\mathbbm{E}X_{jk}-2\sum_{1\leq k\neq l\leq n}\mathbbm{E}X_{ik}\mathbbm{E}X_{jk}\right)\nonumber\\
=&\frac{1}{2n^2(n-1)}\sum_{1\leq k\neq l\leq n}\mathbbm{E}(X_{ik}-X_{il})(X_{jl}-X_{jk}).\label{e_zi_zj}
\end{align}
\subsection{Pre-limiting approximation}
We have the following theorem, comparing the distribution of $\mathbf{Y}_n$ and $\mathbf{D}_n$:
\begin{theorem}\label{combi_pre_lim}
For $\mathbf{Y}_n$ defined in Subsection \ref{intro_section}, $\mathbf{D}_n$ defined in Subsection \ref{section_pre_lim} and any $g\in M^1$, as defined in (\ref{m_1}),
\begin{align*}
&\left|\mathbbm{E}g(\mathbf{Y}_n)-\mathbbm{E}g(\mathbf{D}_n)\right|\\
\leq&\frac{\|g\|_{M^1}}{2n^2(n-1)^2s_n^3}\sum_{1\leq i,j,k,l,u\leq n}\left\lbrace\vphantom{\sum_i^j} \mathbbm{E}|X_{ik}|^3+5\mathbbm{E}|X_{ik}|^2\mathbbm{E}|X_{il}|+9\mathbbm{E}|X_{ik}|^2\mathbbm{E}|X_{jl}|+5\mathbbm{E}|X_{ik}|^2\mathbbm{E}|X_{jk}|\right.\nonumber\\
&+14\mathbbm{E}|X_{ik}|\mathbbm{E}|X_{il}|\mathbbm{E}|X_{jl}| +2\mathbbm{E}|X_{ik}|\mathbbm{E}|X_{il}|\mathbbm{E}|X_{iu}|+6\mathbbm{E}|X_{ik}|\mathbbm{E}|X_{jl}|\mathbbm{E}|X_{iu}|+2\mathbbm{E}[|X_{uk}||X_{ik}|^2]\nonumber\\
&+4\mathbbm{E}|X_{uk}X_{ik}|\mathbbm{E}|X_{il}|+4\mathbbm{E}|X_{uk}X_{ik}|\mathbbm{E}|X_{jl}|+8\mathbbm{E}|X_{uk}|\mathbbm{E}|X_{ik}|\mathbbm{E}|X_{jl}|+2\mathbbm{E}|X_{uk}X_{ik}X_{jk}|\\
&\left.+2\mathbbm{E}|X_{uk}|\mathbbm{E}|X_{ik}|\mathbbm{E}|X_{jk}|+\frac{2}{n}\mathbbm{E}\left(2\left|X_{ik}\right|+6\left|X_{jl}\right|\right)\cdot\left(\sum_{r=1}^n\mathbbm{E}|X_{ir}|^2+\left|\sum_{r=1}^nc_{ir}c_{jr}\right|\right)\right\rbrace.\nonumber\\
&+\frac{2\|g\|_{M^1}}{\sqrt{n}}+\frac{\|g\|_{M^1}}{n^2s_n^2}\sum_{i,j=1}^n\sigma_{ij}^2.
\end{align*}
\end{theorem}
\begin{remark}[Relevance of terms in the bound]
The first long sum in the bound corresponds to $\epsilon_1$ and (to a large extent) $\epsilon_2$ of Theorem \ref{theorem1}. It represents the usual Berry-Esseen third moment estimate arising as a result of applying Taylor's theorem. Term $\frac{2\|g\|_{M^1}}{\sqrt{n}} $ also comes from the estimation of $\epsilon_2$. The last term corresponds to $\epsilon_3$.                                                      
\end{remark}
\begin{remark}
Assuming that $s_n=O(\sqrt{n})$, we obtain that the bound in Theorem \ref{combi_pre_lim} is of order $\frac{1}{\sqrt{n}}$.
\end{remark}
\begin{remark}
If we assume that $\mathbbm{E}|X_{ik}|^3\leq \beta_3$ for all $i,k=1,\cdots, n$ then the bound simplifies in the following way
\begin{align*}
&\left|\mathbbm{E}g(\mathbf{Y}_n)-\mathbbm{E}g(\mathbf{D}_n)\right|
\leq\|g\|_{M^1}\left(\frac{40\beta_3n^2}{(n-1)s_n^3}+\frac{8\beta_3^{1/3}}{n(n-1)s_n^3}\left|\sum_{i,j,r=1}^nc_{ir}c_{jr}\right|+\frac{2}{\sqrt{n}}+\frac{1}{n^2s_n^2}\sum_{i,j=1}^n\sigma_{ij}^2\right).
\end{align*}

\end{remark}
We will use Theorem \ref{theorem1} to prove Theorem \ref{combi_pre_lim}. In the proof, in \textbf{Step 1}, we justify why Theorem \ref{theorem1} may indeed be used in this case. In other words, we check that $\mathbf{D}_n$ of (\ref{pre_lim_an}) satisfies the conditions $\mathbf{D}_n$ of Theorem \ref{theorem1} is supposed to satisfy and that the exchangeable-pair condition for  $\mathbf{Y}_n$ holds.  In \textbf{Step 2} we bound terms $\epsilon_1$ and $\epsilon_3$ coming from Theorem \ref{theorem1}. This is relatively straightforward due to the $\mathbf{Y}_n$ and $\mathbf{Y}_n'$ of Subsection \ref{ex_section} being constructed in such a way that they are close to each other and $R_f$ of the same subsection being small.  Then, in \textbf{Step 3}, we treat the remaining term using a strategy analogous to that of the proof of \cite[Theorem 2.1]{functional_combinatorial}. The strategy is based on Taylor's expansions and considering copies of $\mathbf{Y}_n$ which are independent of some of the summands in $\mathbf{Y}_n$. Finally,  we combine the estimates obtained in the previous steps to obtain the assertion.

\begin{proof}[Proof of theorem \ref{combi_pre_lim}]
We adopt the notation of Subsections \ref{intro_section}, \ref{ex_section} and \ref{section_pre_lim}. Furthermore, we fix a function $g\in M^1$, as defined in (\ref{m_1}) and let $f=\phi_n(g)$, a solution to the Stein equation for $\mathbf{D}_n$, as defined in (\ref{phi}).

\textbf{Step 1.}
We note that $\mathbf{D}_n$ can be expressed in the following way:
$$\mathbf{D}_n=\sum_{i,l=1}^n\left(Z_{il}-\frac{1}{n}\sum_{j=1}^nZ_{jl}\right)J_{i,l},\quad\text{where}\quad J_{i,l}(t)=\frac{X_{il}''}{s_n\sqrt{n-1}}\mathbbm{1}_{[i/n,1]}(t), $$
which, together with (\ref{ex_condition}), lets us apply Theorem \ref{theorem1}.

\textbf{Step 2.}
For the first term in Theorem \ref{theorem1}, for any $g\in M^1$:
$$\epsilon_1=\frac{\|g\|_{M^1}}{12\lambda_n}\mathbbm{E}\|\mathbf{Y}_n-\mathbf{Y}_n'\|^3=\frac{(n-1)\|g\|_{M^1}}{24}\mathbbm{E}\|\mathbf{Y}_n-\mathbf{Y}_n'\|^3.$$
We note that:
\begin{align*}
\mathbbm{E}\|\mathbf{Y}_n-\mathbf{Y}_n'\|^3&\leq\frac{16}{s_n^3}\left(\mathbbm{E}|X_{I\pi(I)}|^3+\mathbbm{E}|X_{J\pi(J)}|^3+\mathbbm{E}|X_{I\pi(J)}|^3+\mathbbm{E}|X_{J\pi(I)}|^3\right)\nonumber\\
&=\frac{16}{n(n-1)s_n^3}\sum_{i\neq j}\left(\mathbbm{E}|X_{i\pi(i)}|^3+\mathbbm{E}|X_{j\pi(j)}|^3+\mathbbm{E}|X_{i\pi(j)}|^3+\mathbbm{E}|X_{j\pi(i)}|^3\right)\nonumber\\
&=\frac{32}{n(n-1)s_n^3}\sum_{i\neq j}\left(\mathbbm{E}|X_{i\pi(i)}|^3+\mathbbm{E}|X_{i\pi(j)}|^3\right)\nonumber\\
&=\frac{64}{n^2s_n^3}\sum_{i,j=1}^n\mathbbm{E}|X_{ij}|^3.
\end{align*}
Hence,
\begin{equation}\label{44.2}
\epsilon_1\leq \frac{8\|g\|_{M^1}}{3ns_n^3}\sum_{i,j=1}^n\mathbbm{E}|X_{ij}|^3.
\end{equation}
Furthermore, by Proposition \ref{prop12.7}:
\begin{align}
\epsilon_3=\left|\frac{1}{ns_n}\sum_{i,j=1}^n\mathbbm{E}Df(\mathbf{Y}_n)\left[X_{i\pi(j)}\mathbbm{1}_{[i/n,1]}\right]\right|=&\left|\frac{1}{ns_n}\sum_{i,j=1}^n\mathbbm{E}Df(\mathbf{Y}_n)\left[X_{ij}\mathbbm{1}_{[i/n,1]}\right]\right|\nonumber\\
\leq&\|g\|_{M^1}\frac{1}{ns_n}\mathbbm{E}\left\|\sum_{i,j=1}^n X_{ij}\mathbbm{1}_{[i/n,1]}\right\|\nonumber\\
\leq&\frac{2\|g\|_{M^1}}{ns_n}\sqrt{\mathbbm{E}\left|\sum_{i,j=1}^n X_{ij}\right|^2}\nonumber\\
\leq&\frac{2\|g\|_{M^1}}{ns_n}\sqrt{\sum_{i,j=1}^n\sigma_{ij}^2}\nonumber\\
\leq&\frac{2\|g\|_{M^1}}{\sqrt{n}},\label{44.3}
\end{align}
where we have used Doob's $L^2$ inequality in the second inequality and (\ref{s_n}) in the last one.

\textbf{Step 3.}
We will follow a strategy similar to that of \cite{chen2015}.
Define a new permutation $\pi_{ijkl}$ coupled with $\pi$, such that:
$$\mathcal{L}(\pi_{ijkl})=\mathcal{L}(\pi|\pi(i)=k,\pi(j)=l),$$
where $\mathcal{L}(\cdot)$ denotes the law.
As noted in \cite[(3.14)]{chen2015}, we can construct it in the following way. For $\tau_{ij}$ denoting the transposition of $i,j$:
$$\pi_{ijkl}=\begin{cases}
\pi,&\text{if }l=\pi(j),k=\pi(i)\\
\pi\cdot\tau_{\pi^{-1}(k),i},&\text{if }l=\pi(j),k\neq \pi(i)\\
\pi\cdot\tau_{\pi^{-1}(l),j},&\text{if }l\neq\pi(j),k=\pi(i)\\
\pi\cdot\tau_{\pi^{-1}(l),i}\cdot\tau_{\pi^{-1}(k),j}\cdot\tau_{ij},&\text{if }l\neq\pi(j),k\neq\pi(i).
\end{cases}$$
We also let 
$$\mathbf{Y}_{n,ijkl}=\frac{1}{s_n}\sum_{i'=1}^nX_{i'\pi_{ijkl}(i')}\mathbbm{1}_{[i'/n,1]}.$$
Then $\mathcal{L}(\mathbf{Y}_{n,ijkl})=\mathcal{L}(\mathbf{Y}_n|\pi(i)=k,\pi(j)=l)$ (recalling that $\mathcal{L}(\cdot)$ denotes the law). Also, for each choice of $i\neq j$, $k\neq l$ let $\mathbbm{X}^{ijkl}:=\left\lbrace X_{i'j'}^{ijkl}:i',j'\in[n]\right\rbrace$ be the same as $\mathbbm{X}:=\lbrace X_{ij};i,j\in[n]\rbrace$ except that $\lbrace X_{ik},X_{il},X_{jk},X_{jl}\rbrace$ has been replaced by an independent copy $\lbrace X_{ik}',X_{il}',X_{jk}',X_{jl}'\rbrace$. Then let 
$$\mathbf{Y}_n^{ijkl}=\frac{1}{s_n}\sum_{i'=1}^nX_{i'\pi(i')}^{ijkl}\mathbbm{1}_{[i'/n,1]}$$
and note that $\mathbf{Y}_n^{ijkl}$ is independent of $\lbrace X_{ik},X_{il},X_{jk},X_{jl}\rbrace$ and $\mathcal{L}(\mathbf{Y}_n^{ijkl})=\mathcal{L}(\mathbf{Y}_n)$ (where $\mathcal{L}$ denotes the law).

Now, by Lemma \ref{lemma1_app}, proved in the appendix, for $\epsilon_2$ of Theorem \ref{theorem1},
\begin{align}
\epsilon_2=&\left|\frac{1}{2\lambda_n}\mathbbm{E}D^2f(\mathbf{Y}_n)\left[(\mathbf{Y}_n-\mathbf{Y}_n'),\mathbf{Y}_n-\mathbf{Y}_n'\right]-\mathbbm{E}D^2f(\mathbf{Y}_n)[\mathbf{D}_n,\mathbf{D}_n]\right|\nonumber\\
\leq & A+B\label{44.4}
\end{align}
where
\begin{align}
A=&\left|\frac{1}{n(n-1)s_n^2}\underset{i\neq j,k\neq l}{\sum_{1\leq i,j,k,l\leq n}}\mathbbm{E}\left\lbrace\left[\frac{(X_{ik}-X_{il})^2}{2n}-\frac{\hat{Z}_i^2}{n-1}\right]\right.
\left(D^2f(\mathbf{Y}_{n,ijkl})-D^2f\left(\mathbf{Y}_{n}^{ijkl}\right)\right)\left[\mathbbm{1}_{[i/n,1]}\mathbbm{1}_{[i/n,1]}\right]\vphantom{\left[\frac{1^2}{2^2}\right]}\right\rbrace\nonumber\\
&+\frac{1}{n(n-1)s_n^2}\underset{i\neq j,k\neq l}{\sum_{1\leq i,j,k,l\leq n}}\mathbbm{E}\left\lbrace\left[\frac{(X_{ik}-X_{il})(X_{jl}-X_{jk})}{2n}-\hat{Z}_i\hat{Z}_j\right]\right.\nonumber\\
&\left.\left.\hphantom{\frac{1}{n(n-1)s_n^2}\underset{i\neq j,k\neq l}{\sum_{1\leq i,j,k,l\leq n}}\mathbbm{E}}\cdot\left(D^2f(\mathbf{Y}_{n,ijkl})-D^2f\left(\mathbf{Y}_{n}^{ijkl}\right)\right)\left[\mathbbm{1}_{[i/n,1]},\mathbbm{1}_{[j/n,1]}\right]\vphantom{\left[\frac{1^2}{2^2}\right]}\right\rbrace\vphantom{\frac{1}{2}\underset{2}{\sum_{2}}}\right|\label{a_def},\\
B=&\left|\frac{1}{n(n-1)s_n^2}\underset{i\neq j,k\neq l}{\sum_{1\leq i,j,k,l\leq n}}\mathbbm{E}\left\lbrace\left[\frac{(X_{ik}-X_{il})^2}{2n}-\frac{\hat{Z}_i^2}{n-1}\right]D^2f\left(\mathbf{Y}_{n}^{ijkl}\right)\left[\mathbbm{1}_{[i/n,1]}\mathbbm{1}_{[i/n,1]}\right]\right\rbrace\right.\nonumber\\
&\left.+\frac{1}{n(n-1)s_n^2}\underset{i\neq j,k\neq l}{\sum_{1\leq i,j,k,l\leq n}}\mathbbm{E}\left\lbrace\left[\frac{(X_{ik}-X_{il})(X_{jl}-X_{jk})}{2n}-\hat{Z}_i\hat{Z}_j\right]\right.\right.\nonumber\\
&\phantom{\frac{1}{n(n-1)s_n^2}\underset{i\neq j,k\neq l}{\sum_{1\leq i,j,k,l\leq n}}\mathbbm{E}}\left.\left.\cdot D^2f\left(\mathbf{Y}_{n}^{ijkl}\right)\left[\mathbbm{1}_{[i/n,1]},\mathbbm{1}_{[j/n,1]}\right]\vphantom{\left[\frac{1^2}{2^2}\right]}\right\rbrace\vphantom{\frac{1^2}{2}\underset{2}{\sum_{2}}}\right|.\nonumber
\end{align}
Recalling that $\mathbf{Y}_n^{ijkl}$ is independent of $\lbrace X_{ik},X_{il},X_{jk},X_{jl}\rbrace$ and $\mathcal{L}(\mathbf{Y}_n^{ijkl})=\mathcal{L}(\mathbf{Y}_n)$,
\begin{align}
B=&\left|\frac{1}{n(n-1)s_n^2}\underset{i\neq j,k\neq l}{\sum_{1\leq i,j,k,l\leq n}}\mathbbm{E}\left[\frac{(X_{ik}-X_{il})^2}{2n}-\frac{\hat{Z}_i^2}{n-1}\right]\mathbbm{E}\left\lbrace D^2f\left(\mathbf{Y}_{n}\right)\left[\mathbbm{1}_{[i/n,1]}\mathbbm{1}_{[i/n,1]}\right]\right\rbrace\right.\nonumber\\
&+\frac{1}{n(n-1)s_n^2}\underset{i\neq j,k\neq l}{\sum_{1\leq i,j,k,l\leq n}}\mathbbm{E}\left[\frac{(X_{ik}-X_{il})(X_{jl}-X_{jk})}{2n}-\hat{Z}_i\hat{Z}_j\right]
\left.\mathbbm{E}\left\lbrace D^2f\left(\mathbf{Y}_{n}\right)\left[\mathbbm{1}_{[i/n,1]},\mathbbm{1}_{[j/n,1]}\right]\vphantom{\frac{1}{2}}\right\rbrace\vphantom{\frac{1^2}{2^2}\underset{2}{\sum_2}}\right|\nonumber\\
\leq&\frac{\|g\|_{M^1}}{n^2(n-1)s_n^2}\sum_{1\leq i\neq j\leq n}\sum_{r=1}^n\sigma^2_{ir}\nonumber\\
=&\frac{\|g\|_{M^1}}{n^2s_n^2}\sum_{i,j=1}^n\sigma_{ij}^2,\label{44.5}
\end{align}
where the inequality follows by (\ref{e_zi}), (\ref{e_zi_zj}) and Proposition \ref{prop12.7}.
Furthermore, by Lemma \ref{lemma2_app}, proved in the appendix,
\begin{align}
A\leq&\frac{\|g\|_{M^1}}{2n^2(n-1)^2s_n^3}\sum_{1\leq i,j,k,l,u\leq n}\left\lbrace\vphantom{\sum_i^j} \mathbbm{E}|X_{ik}|^3+5\mathbbm{E}|X_{ik}|^2\mathbbm{E}|X_{il}|+9\mathbbm{E}|X_{ik}|^2\mathbbm{E}|X_{jl}|+5\mathbbm{E}|X_{ik}|^2\mathbbm{E}|X_{jk}|\right.\nonumber\\
&+14\mathbbm{E}|X_{ik}|\mathbbm{E}|X_{il}|\mathbbm{E}|X_{jl}| +2\mathbbm{E}|X_{ik}|\mathbbm{E}|X_{il}|\mathbbm{E}|X_{iu}|+6\mathbbm{E}|X_{ik}|\mathbbm{E}|X_{jl}|\mathbbm{E}|X_{iu}|+2\mathbbm{E}[|X_{uk}||X_{ik}|^2]\nonumber\\
&+4\mathbbm{E}|X_{uk}X_{ik}|\mathbbm{E}|X_{il}|+4\mathbbm{E}|X_{uk}X_{ik}|\mathbbm{E}|X_{jl}|+8\mathbbm{E}|X_{uk}|\mathbbm{E}|X_{ik}|\mathbbm{E}|X_{jl}|+2\mathbbm{E}|X_{uk}X_{ik}X_{jk}|\nonumber\\
&\left.+2\mathbbm{E}|X_{uk}|\mathbbm{E}|X_{ik}|\mathbbm{E}|X_{jk}|+\frac{2}{n}\mathbbm{E}\left(2\left|X_{ik}\right|+6\left|X_{jl}\right|\right)\cdot\left(\sum_{r=1}^n\mathbbm{E}|X_{ir}|^2+\left|\sum_{r=1}^nc_{ir}c_{jr}\right|\right)\right\rbrace.\label{a_est}
\end{align}
We now use (\ref{44.2}),(\ref{44.3}),(\ref{44.4}),(\ref{44.5}),(\ref{a_est}) to obtain the assertion.
\end{proof}

\subsection{Convergence to a continuous Gaussian process}
\begin{theorem}\label{conv_cont_theorem}
Let $\mathbbm{X}$ and $\mathbf{Y}_n$ be as defined in Subsection \ref{intro_section} and suppose that for all $u,t\in[0,1]$,
\begin{equation}\label{assumption1}
\frac{1}{s_n^2(n-1)}\sum_{i=1}^{\lfloor nt\rfloor}\sum_{j=1}^{\lfloor nu\rfloor}\sum_{k=1}^n\mathbbm{E}X_{ik}X_{jk}\left(\delta_{i,j}-\frac{1}{n}\right)\xrightarrow{n\to\infty}\sigma(u,t)
\end{equation}
and
\begin{equation}\label{assumption2}
\frac{1}{s_n^2}\sum_{i=1}^{\lfloor nt\rfloor}\sum_{j=1}^{\lfloor nu\rfloor}\sum_{l=1}^n\mathbbm{E}X_{il}X_{jl}\xrightarrow{n\to\infty}\sigma^{(2)}(u,t)
\end{equation}
pointwise for some functions $\sigma,\sigma^{(2)}:[0,1]^2\to\mathbbm{R}_+$. Suppose furthermore that
\begin{equation}\label{assumption3}
\sup_{n\in\mathbbm{N}}\frac{1}{n^2s_n^4}\sum_{l=1}^n\sum_{i=1}^n\text{Var}\left[X_{il}^2\right]<\infty.
\end{equation}
and
\begin{equation}\label{assumption4}
\frac{1}{s_n^2(n-1)}\sum_{i=1}^{\lfloor nt\rfloor}\left(\sum_{l=1}^nX_{il}''Z_{il}\right)^2\xrightarrow{P} c(t)
\end{equation}
pointwise for some function $c:[0,1]\to\mathbbm{R}_+$ and
\begin{equation}\label{assumption5}
\lim_{n\to\infty}\frac{1}{s_n\sqrt{n-1}}\mathbbm{E}\left[\sup_{i=1,\cdots,n}|X_{il}''Z_{il}|\right]=0.
\end{equation}
Then $\left(\mathbf{Y}_n(t),t\in[0,1]\right)$ converges weakly in the uniform topology to a continuous Gaussian process $\left(\mathbf{Z}(t),t\in[0,1]\right)$ with the covariance function $\sigma$.
\end{theorem}
\begin{remark}
Assumption (\ref{assumption2}) could also say that
$$\frac{1}{s_n^2}\sum_{i=1}^{\lfloor nt\rfloor}\sum_{j=1}^{\lfloor nu\rfloor}\sum_{l=1}^n\mathbbm{E}X_{il}X_{jl}$$
simply converges pointwise rather than giving the limit a name. However, we will use $\sigma^{(2)}$ in the proof so it is convenient to use it in the formulation of the Theorem as well.
\end{remark}
\begin{remark}
Assumption (\ref{assumption4}) is necessary for the limiting process in Theorem \ref{conv_cont_theorem} to be continuous. It essentially corresponds to the the assumption that the quadratic variation of the following process $$\mathbf{D}_n^{(1)}(t)=\frac{1}{s_n\sqrt{n-1}}\sum_{i=1}^{\lfloor nt\rfloor}\sum_{l=1}^nX_{il}''Z_{il}$$ converges to the function $c$ pointwise in probability, which then implies the weak convergence of the process $\mathbf{D}_n^{(1)}$ to a continuous process. While it is relatively easy to show that $\mathbf{D}_n^{(2)}=\mathbf{D}_n-\mathbf{D}_n^{(1)}$ converges to a continuous limit, we had to explicitly add this assumption to ensure that $\mathbf{D}_n$ does as well.
\end{remark}
The proof of Theorem \ref{conv_cont_theorem} will be similar to the proof of \cite[Theorem 3.3]{functional_combinatorial}. The pre-limiting approximand $\mathbf{D}_n$, defined in Subsection \ref{section_pre_lim}, will be expressed as a sum of two parts. In \textbf{Steps 1} and \textbf{2} we prove that each of those parts is C-tight (i.e. they are tight and for each of them any convergent subsequence converges to a process with continuous sample paths). In \textbf{Step 3} we show that the assumptions of Theorem \ref{conv_cont_theorem} trivially imply the convergence of the covariance function of $\mathbf{D}_n$, which together with C-tightness implies the convergence of $\mathbf{D}_n$ to a continuous process. Theorem \ref{theorem_pre_limiting} will then be combined with Proposition \ref{prop_m} to show convergence of $\mathbf{Y}_n$ to the same limiting process. Finally, the combinatorial central limit theorem for random arrays, proved in \cite{chen_ho} and analysed in \cite{chen2015}, will imply that $\mathbf{Z}$ is Gaussian.

\begin{proof}[Proof of Theorem \ref{conv_cont_theorem}]
We will use the notation of Subsections \ref{intro_section} and \ref{section_pre_lim}.

\textbf{Step 1.}
Note that $\mathbf{D}_n=\mathbf{D}_n^{(1)}+\mathbf{D}_n^{(2)}$, where:
$$\mathbf{D}_n^{(1)}(t)=\frac{1}{s_n\sqrt{n-1}}\sum_{i=1}^{\lfloor nt\rfloor}\sum_{l=1}^nX_{il}''Z_{il},\quad \mathbf{D}_n^{(2)}(t)=\frac{1}{s_n\sqrt{n-1}}\sum_{i=1}^{\lfloor nt\rfloor}\sum_{l=1}^nX_{il}''\bar{Z}_l$$
for $\bar{Z}_l=\frac{1}{n}\sum_{j=1}^nZ_{jl}$.

Now, note that, by (\ref{assumption4}):
$$\left< \mathbf{D}_n^{(1)}\right>_t\xrightarrow{P}c(t)$$
pointwise, where $\left<\cdot\right>$ denotes quadratic variation.
Therefore, by \cite[Chapter 7, Theorem 1.4]{ethier} and using (\ref{assumption5}), we obtain that $\mathbf{D}_n^{(1)}$ converges weakly in the Skorokhod topology on $D[0,1]$ to a continuous Gaussian process with independent increments.

We now note that the Skorokhod space equipped with the metric (topologically equivalent to the Skorokhod metric) with respect to which it is complete. It is also universally measurable by the discussion at the beginning of \cite[Chapter 11.5]{dudley_book}. Since it is also separable and $\mathbf{D}_n^{(1)}\Rightarrow \mathbf{Z}_1$, for some continuous process $\mathbf{Z}_1$, in the Skorokhod topology, \cite[Theorem 11.5.3]{dudley_book} implies that $(\mathbf{D}_n^{(1)})_{n\geq 1}$ is C-tight. 

\textbf{Step 2.}
Also, note that for $u>t$ s.t. $\lfloor nu\rfloor\geq \lfloor nt\rfloor +1$,
\begin{align*}
\mathbbm{E}\left[\left.\left|\mathbf{D}_n^{(2)}(u)-\mathbf{D}_n^{(2)}(t)\right|^2\right|X_{il}'',i,l\in[n]\right]=&\frac{1}{n(n-1)s_n^2}\sum_{l=1}^n\left(\sum_{i=\lfloor nt\rfloor+1}^{\lfloor nu\rfloor}X_{il}\right)^2\\
\leq&\frac{\lfloor nu\rfloor -\lfloor nt\rfloor}{n(n-1)s_n^2}\sum_{l=1}^n\sum_{i=\lfloor nt\rfloor+1}^{\lfloor nu\rfloor}X_{il}^2
\end{align*}
and
$$\mathbbm{E}\left[\left.\left|\mathbf{D}_n^{(2)}(u)-\mathbf{D}_n^{(2)}(t)\right|^2\right|X_{il}'',i,l\in[n]\right]=0,\quad\text{for }u>t\text{ s.t. }\lfloor nu\rfloor=\lfloor nt\rfloor.$$
Since $\left(\left.\mathbf{D}_n^{(2)}\right|X_{il}'',i,l\in[n]\right)$ is Gaussian for $u$, such that $\lfloor nu\rfloor \geq \lfloor nt\rfloor +1$,
\begin{align}
\mathbbm{E}\left|\mathbf{D}_n^{(2)}(u)-\mathbf{D}_n^{(2)}(t)\right|^4
=&3\mathbbm{E}\left\lbrace\left(\mathbbm{E}\left[\left.\left|\mathbf{D}_n^{(2)}(u)-\mathbf{D}_n^{(2)}(t)\right|^2\right|X_{il}'',i,l\in[n]\right]\right)^2\right\rbrace\nonumber\\
\leq&3\left(\frac{\lfloor nu\rfloor -\lfloor nt\rfloor}{n(n-1)s_n^2}\right)^2\mathbbm{E}\left(\sum_{l=1}^n\sum_{i=\lfloor nt\rfloor+1}^{\lfloor nu\rfloor}X_{il}^2\right)^2\nonumber\\
=&3\left(\frac{\lfloor nu\rfloor -\lfloor nt\rfloor}{n(n-1)s_n^2}\right)^2\left[\left(\sum_{l=1}^n\sum_{i=\lfloor nt\rfloor+1}^{\lfloor nu\rfloor}\mathbbm{E}X_{il}^2\right)^2+\sum_{l=1}^n\sum_{i=\lfloor nt\rfloor+1}^{\lfloor nu\rfloor}\left(\mathbbm{E}X_{il}^4-\left(\mathbbm{E}X_{il}^2\right)^2\right)\right]\nonumber\\
\leq&C\left(\frac{\lfloor nu\rfloor -\lfloor nt\rfloor}{(n-1)}\right)^2\label{4.1}
\end{align}
for some constant $C$, by (\ref{assumption3}).
Now, note that:
$$\text{Cov}\left(\mathbf{D}_n^{(2)}(t),\mathbf{D}_n^{(2)}(u)\right)=\frac{1}{s_n^2}\sum_{i=1}^{\lfloor nt\rfloor}\sum_{j=1}^{\lfloor nu\rfloor}\sum_{l=1}^n\mathbbm{E}X_{il}X_{jl}\xrightarrow{n\to\infty}\sigma^{(2)}(t,u),$$
by (\ref{assumption2}).
Consider a mean zero Gaussian process $\mathbf{Z}_2$ with covariance function $\mathbbm{E}\mathbf{Z}_2(t)\mathbf{Z}_2(u)=\sigma^{(2)}(t,u)$. The finite dimensional distributions of $\mathbf{D}_n^{(2)}$ converge to those of $\mathbf{Z}_2$.
We can now construct $\mathbf{D}_n^{(2)}$ and $\mathbf{Z}_2$ on the same probability space and use Skorokhod's representation theorem, Fatou's lemma and (\ref{4.1}) to conclude that:
$$\mathbbm{E}\left(\left|\mathbf{Z}_{2}(u)-\mathbf{Z}_2(t)\right|^4\right)\leq \lim_{n\to\infty}\mathbbm{E}\left(\left|\mathbf{D}_n^{(2)}(u)-\mathbf{D}_n^{(2)}(t)\right|^4\right)\leq C(u-t)^2.$$
By \cite[Theorem 12.4]{billingsley1}, we can assume that $\mathbf{Z}_2\in C[0,1]$. Now, note that for $0\leq t\leq v\leq u\leq 1$:
\begin{align*}
\mathbbm{E}\left|\mathbf{D}_n^{(2)}(v)-\mathbf{D}_n^{(2)}(t)\right|^2\left|\mathbf{D}_n^{(2)}(v)-\mathbf{D}_n^{(2)}(u)\right|^2
\leq& \sqrt{\mathbbm{E}\left|\mathbf{D}_n^{(2)}(v)-\mathbf{D}_n^{(2)}(t)\right|^4\mathbbm{E}\left|\mathbf{D}_n^{(2)}(v)-\mathbf{D}_n^{(2)}(u)\right|^4}\\
\stackrel{(\ref{4.1})}\leq& C \frac{(\lfloor nv\rfloor -\lfloor nt\rfloor)(\lfloor nu\rfloor-\lfloor nv\rfloor)}{(n-1)^2}\\
\leq&\bar{C}(u-t)^2;
\end{align*}
for some constant $\bar{C}$. Therefore, by \cite[Theorem 15.6]{billingsley1}, $\mathbf{D}_n^{(2)}\Rightarrow \mathbf{Z}_2$ in the Skorokhod and uniform topologies and so, by \cite[Theorem 11.5.3]{dudley_book}, $\mathbf{D}_n^{(2)}$ is C-tight.

\textbf{Step 3.}
Since both $\mathbf{D}_n^{(1)}$ and $\mathbf{D}_n^{(2)}$ are C-tight, so is their difference $\mathbf{D}_n$. Now:
\begin{align*}
\text{Cov}(\mathbf{D}_n(t),\mathbf{D}_n(u))=&\frac{1}{s_n^2(n-1)}\sum_{i=1}^{\lfloor nt\rfloor}\sum_{j=1}^{\lfloor nu\rfloor}\sum_{k,l=1}^n\mathbbm{E}\left\lbrace X_{ik}X_{jl}\left(Z_{ik}-\bar{Z}_k\right)\left(Z_{jl}-\bar{Z}_l\right)\right\rbrace\\
=&\frac{1}{s_n^2(n-1)}\sum_{i=1}^{\lfloor nt\rfloor}\sum_{j=1}^{\lfloor nu\rfloor}\sum_{k=1}^n\mathbbm{E}\left\lbrace X_{ik}X_{jk}\left(Z_{ik}-\bar{Z}_k\right)\left(Z_{jk}-\bar{Z}_k\right)\right\rbrace\\
=&\frac{1}{s_n^2(n-1)}\sum_{i=1}^{\lfloor nt\rfloor}\sum_{j=1}^{\lfloor nu\rfloor}\sum_{k=1}^n\mathbbm{E}X_{ik}X_{jk}\left(\delta_{i,j}-\frac{1}{n}\right)\xrightarrow{n\to\infty}\sigma(u,t),
\end{align*}
by (\ref{assumption1}) and we obtain that $\mathbf{D}_n$ converges to a random element $\mathbf{Z}\in C[0,1]$ with covariance function $\sigma$ in distribution with respect to the uniform and Skorokhod topologies.

Proposition \ref{prop_m} and Theorem \ref{combi_pre_lim} therefore imply that $\left(\mathbf{Y}_n(t),t\in[0,1]\right)$ converges weakly to $\left(\mathbf{Z}(t),t\in[0,1]\right)$ in the uniform topology. Using, for example, \cite[Theorem 1.1]{chen2015}, we conclude that $\mathbf{Z}$ is a Gaussian process.
\end{proof}

\section{Edge and two-star counts in Bernoulli random graphs}\label{section6}
In this section we consider a process representing a properly rescaled number of edges  in a Bernoulli random graph with a fixed edge probability and $\lfloor nt\rfloor$ edges for $t\in[0,1]$. A similar setup has been considered in \cite{reinert_roellin1}, where the authors established a bound on the distance between a three-dimensional vector consisting of a rescaled number of edges, a rescaled number of two-stars and a rescaled number of triangles in a $G(n,p)$ graph and a three-dimensional Gaussian vector. We first compare our process to a pre-limiting Gaussian processes with paths in $D([0,1],\mathbbm{R})$ and bound the distance between the two in Theorem \ref{theorem_pre_limiting}. Then, in Theorem \ref{theorem_continuous}, we bound the distance of our process from a continuous Gaussian process.
\subsection{Introduction}\label{section1}
Let us consider a Bernoulli random graph $G(n,p)$ on $n$ vertices with edge probabilities $p$.

Let $I_{i,j}=I_{j,i}$ be the Bernoulli$(p)$-indicator that edge $(i,j)$ is present in this graph. These indicators, for $(i,j)\in\lbrace 1,\cdots,n\rbrace^2$ are independent. We will look at a process representing at each $t\in[0,1]$ the re-scaled total number of edges in the graph formed out of the given Bernoulli random graph by considering only its first $\lfloor nt\rfloor$ vertices and the edges between them:
$$\mathbf{T}_n(t)=\frac{\lfloor nt\rfloor-2}{2n^2}\sum_{i,j=1}^{\lfloor nt\rfloor} I_{i,j}=\frac{\lfloor nt\rfloor-2}{n^2}\sum_{1\leq i<j\leq \lfloor nt\rfloor}I_{i,j}.$$
Let $\mathbf{Y}_n(t)=\mathbf{T}_n(t)-\mathbbm{E}\mathbf{T}_n(t)$ for $t\in[0,1]$.
\begin{remark}
Note that, for all $t\in[0,1]$, $\mathbbm{E}\mathbf{T}_n(t)=\frac{\lfloor nt\rfloor -2}{n^2}{\lfloor nt\rfloor \choose 2}p$. Furthermore, note that, by an argument similar to that of \cite[Section 5]{reinert_roellin1}, the variance of $\mathbf{T}_n(t)-\mathbbm{E}\mathbf{T}_n(t)$ is given by
$$3\frac{(\lfloor nt\rfloor -2){\lfloor nt\rfloor \choose 3}}{n^4}p(1-p).$$
\end{remark}
\begin{remark}
While one might often be interested in the asymptotic properties of the number of two-stars or the number of triangles in the above mentioned Bernoulli random graph, it is difficult to prove useful exchangeable-pair conditions for those. Specifically, one encounters several problems which were overcome in \cite{reinert_roellin1} (in the finite-dimensional setting) with the so-called \textit{embedding method} (introduced in \cite[Section 3]{reinert_roellin1}). Indeed, the problem becomes easier when the distribution of the number of edges, the number of two-stars and the number of triangles is considered jointly rather than considering the one-dimensional marginal laws separately. We therefore leave the problem of establishing bounds on the functional approximations of those statistics for further research, related to the multivariate method of exchangeable pairs (cf. Remark \ref{uni_multi}).
\end{remark}
\subsection{Exchangeable pair setup}\label{section2}
We now construct an exchangeable pair, as in \cite{reinert_roellin1}, by picking $(I,J)$ according to $\mathbbm{P}[I=i,J=j]=\frac{1}{{n\choose 2}}$ for $1\leq i<j\leq n$. If $I=i,J=j$, we replace $I_{i,j}=I_{j,i}$ by an independent copy $I_{i,j}'=I_{j,i}'$ and put:
\begin{align*}
\mathbf{T}_n'(t)&=\mathbf{T}_n(t)-\frac{\lfloor nt\rfloor-2}{n^2}\left(I_{I,J}-I_{I,J}'\right)\mathbbm{1}_{[I/n,1]\cap[J/n,1]}(t).
\end{align*}
We also let $\mathbf{Y}_n'(t)=\mathbf{T}_n'(t)-\mathbbm{E}\mathbf{T}_n(t)$ and note that, for $\mathbf{Y}_n=\left(\mathbf{Y}_n(t),t\in[0,1]\right)$ and $\mathbf{Y}_n'=\left(\mathbf{Y}_n'(t),t\in[0,1]\right)$, $(\mathbf{Y}_n,\mathbf{Y}_n')$ forms an exchangeable pair.
We note that for any $f\in M$, as defined in Section \ref{section22},
\begin{align*}
&\mathbbm{E}^{\mathbf{Y}_n}\left\lbrace Df(\mathbf{Y}_n)\left[\mathbf{T}_n'-\mathbf{T}_n\right]\right\rbrace\\
=&\mathbbm{E}^{\mathbf{Y}_n}\left\lbrace Df(\mathbf{Y}_n)\left[\frac{\lfloor n\cdot\rfloor-2}{n^2}\left(I_{I,J}'-I_{I,J}\right)\mathbbm{1}_{[I/n,1]\cap[J/n,1]}\right]\right\rbrace\\
=&\frac{2}{n^3(n-1)}\sum_{i<j}\mathbbm{E}^{\mathbf{Y}_n}\left\lbrace Df(\mathbf{Y}_n)\left[(\lfloor n\cdot\rfloor-2)\left(I_{i,j}'-I_{i,j}\right)\mathbbm{1}_{[i/n,1]\cap[j/n,1]}\right]|I=i,J=j\right\rbrace\\
=&-\frac{1}{{n\choose 2}}Df(\mathbf{Y}_n)[\mathbf{T}_n]+\frac{2}{n^3(n-1)}p\sum_{i<j}Df(\mathbf{Y}_n)\left[(\lfloor n\cdot\rfloor-2)\mathbbm{1}_{[i/n,1]\cap[j/n,1]}\right]\\
=&-\frac{1}{{n\choose 2}}Df(\mathbf{Y}_n)[\mathbf{T}_n(\cdot)-\mathbbm{E}\mathbf{T}_n(\cdot)]
\end{align*}
and so
$$\mathbbm{E}^{\mathbf{Y}_n}Df(\mathbf{Y}_n)\left[\mathbf{Y}_n'-\mathbf{Y}_n\right]=-\lambda_nDf(\mathbf{Y}_n)[\mathbf{Y}_n],$$
where
\begin{equation}\label{eq_lambda}
\lambda_n=\frac{2}{n(n-1)}.
\end{equation}
Therefore, condition (\ref{condition}) is satisfied with $\lambda_n$ of (\ref{eq_lambda}) and $R_f=0$.
\subsection{A pre-limiting process}\label{section3}
Let $Z_{ii}=0$ for all $i\in[n]=\lbrace 1,\dots,n\rbrace$. Furthermore, assume that the collection $\lbrace Z_{ij}: i,j\in [n], i\neq j\rbrace$ is jointly centred Gaussian with the following covariance structure:
\begin{align*}
&\mathbbm{E}Z_{ij}Z_{kl}=\begin{cases}
\frac{p(1-p)}{2n^4},&i=k,j=l,i\neq j\\
0,&\text{otherwise.}\end{cases}
\end{align*}
Let 
\begin{align*}
&\mathbf{D}_n=\left(\lfloor nt\rfloor-2\right)\sum_{i,j=1}^{\lfloor nt\rfloor}Z_{ij},\quad t\in[0,1].
\end{align*}
\subsection{Distance from the pre-limiting process}
We first give a theorem providing a bound on the distance between $\mathbf{Y}_n$ and the pre-limiting piecewise constant Gaussian process.
\begin{theorem}\label{theorem_pre_limiting}
Let $\mathbf{Y}_n$ be defined as in Section \ref{section1} and $\mathbf{D}_n$ be defined as in Section \ref{section3}. Then, for any $g\in M$,
$$\left|\mathbbm{E}g(\mathbf{Y}_n)-\mathbbm{E}g(\mathbf{D}_n)\right|\leq\frac{1}{8}\|g\|_{M}n^{-1}.$$
\end{theorem}
In \textbf{Step 1} of the proof, which is based on Theorem \ref{theorem1}, we estimate term $\epsilon_1$ thereof. It involves bounding the third moment of $\|\mathbf{Y}_n-\mathbf{Y}_n'\|$ for $\mathbf{Y}_n'$ constructed in Section \ref{section2}. In \textbf{Step 2} we treat $\epsilon_2$, which requires some more involved calculations. Term $\epsilon_3$ is equal to zero as $R_f$ of Section \ref{section2} is equal to zero.
\begin{proof}[Proof of Theorem \ref{theorem_pre_limiting}]
We adopt the notation of sections \ref{section1}, \ref{section2}, \ref{section3}. We will apply Theorem \ref{theorem1}. Let $g\in M$ and let $f=\phi(g)$, as defined in (\ref{phi}).

\textbf{Step 1.}
 First note that
\begin{align*}
&\mathbbm{E}\left[\|\mathbf{Y}_n-\mathbf{Y}_n'\|^3\right]
\leq \mathbbm{E}\left[\frac{(n-2)^3}{n^6}\left|I_{I,J}-I'_{I,J}\right|^3\right]\leq \frac{(n-2)^3}{n^6}
\end{align*}
where the second inequality follows because $|I_{I,J}-I_{I,J}'|\leq 1$. Therefore,
\begin{align}\label{4.1.1.1}
\epsilon_1\leq\frac{\|g\|_{M}}{12\lambda_n}\mathbbm{E}\left[\|\mathbf{Y}_n-\mathbf{Y}_n'\|^3\right]\leq \frac{\|g\|_{M}}{24n}.
\end{align}

\textbf{Step 2.}
For $\epsilon_2$ in Theorem \ref{theorem1}, we wish to bound
\begin{align}
\epsilon_2=&\left|\frac{1}{2\lambda_n}\mathbbm{E}D^2f(\mathbf{Y}_n)\left[\mathbf{Y}_n-\mathbf{Y}_n',\mathbf{Y}_n-\mathbf{Y}_n'\right]-\mathbbm{E}D^2f(\mathbf{Y}_n)\left[\mathbf{D}_n,\mathbf{D}_n\right]\right|\nonumber\\
=&\left|\frac{n(n-1)}{4}\mathbbm{E}D^2f(\mathbf{Y}_n)\left[\mathbf{T}_n-\mathbf{T}_n',\mathbf{T}_n-\mathbf{T}_n'\right]-\mathbbm{E}D^2f(\mathbf{Y}_n)\left[\mathbf{D}_n,\mathbf{D}_n\right]\right|.\nonumber
\end{align}
For fixed $i,j\in\lbrace 1,\cdots, n\rbrace$, let $\mathbf{Y}_n^{ij}$ be equal to $\mathbf{Y}_n$ except for the fact that $I_{ij}$ is replaced by an independent copy, i.e. for all $t\in [0,1]$ let:
\begin{align*}
\mathbf{T}_n^{ij}(t)&=\mathbf{T}_n(t)-\frac{\lfloor nt\rfloor-2}{n^2}\left(I_{ij}-I_{ij}'\right)\mathbbm{1}_{[i/n,1]\cap[j/n,1]}(t)
\end{align*}
and let $\mathbf{Y}_n^{ij}(t)=\mathbf{T}_n^{ij}(t)-\mathbbm{E}\mathbf{T}_n(t)$.

By noting that the mean zero $Z_{ik}$ and $Z_{i'j}$ are independent for $i\neq i'$ or $j\neq k$, we obtain:
\begin{align}
\epsilon_2=&\left|\vphantom{\sum_1^1}\frac{n(n-1)}{4}\mathbbm{E}D^2f(\mathbf{Y}_n)\left[\mathbf{T}_n-\mathbf{T}_n',\mathbf{T}_n-\mathbf{T}_n'\right]\right.\\
&\left.-\sum_{j,k=1}^n\mathbbm{E}D^2f(\mathbf{Y}_n)\left[\sum_{i=1}^{n}Z_{ik}(\lfloor n\cdot \rfloor -2)\mathbbm{1}_{[i/n,1]\cap[k/n,1]},\sum_{i=1}^{n}Z_{ij}(\lfloor n\cdot \rfloor -2)\mathbbm{1}_{[i/n,1]\cap[j/n,1]}\right]\right|\nonumber\\
=&\left|\frac{1}{4n^4}\sum_{1\leq i\neq j\leq n}\mathbbm{E}\left\lbrace\left(I_{ij}-2pI_{ij}+p\right)D^2f(\mathbf{Y}_n)\left[(\lfloor n\cdot \rfloor -2)\mathbbm{1}_{[i/n,1]\cap[j/n,1]},(\lfloor n\cdot \rfloor -2)\mathbbm{1}_{[i/n,1]\cap[j/n,1]}\right]\right\rbrace\right.\nonumber\\
&-\left.\sum_{i,j=1}^n\left\lbrace\mathbbm{E}\left(Z_{ij}\right)^2\mathbbm{E}D^2f(\mathbf{Y}_n)\left[(\lfloor n\cdot \rfloor -2)\mathbbm{1}_{[i/n,1]\cap[j/n,1]},(\lfloor n\cdot \rfloor -2)\mathbbm{1}_{[i/n,1]\cap[j/n,1]}\right]\vphantom{\left(Z^{1}_2\right)^2}\right\rbrace\vphantom{\frac{1}{4n^2}\sum_{i}}\right|\nonumber\\
=&\left|\sum_{1\leq i\neq j\leq n}\mathbbm{E}\left\lbrace \left(\frac{1}{4n^4}(I_{ij}-2pI_{ij}+p)-\mathbbm{E}\left(Z_{ij}\right)^2\right)\right.\right.\nonumber\\
&\left.\left.\hphantom{\sum_{1\leq i\neq j\leq n}\mathbbm{E}}\cdot D^2f(\mathbf{Y}_n)\left[(\lfloor n\cdot \rfloor -2)\mathbbm{1}_{[i/n,1]\cap[j/n,1]},(\lfloor n\cdot \rfloor -2)\mathbbm{1}_{[i/n,1]\cap[j/n,1]}\right]\vphantom{\frac{1}{4n^2}}\right\rbrace\vphantom{\sum_1^2}\right|\nonumber\\
=&\left|\sum_{1\leq i\neq j\leq n}\mathbbm{E}\left\lbrace \frac{1}{4n^4}(I_{ij}-2pI_{ij}+p)\right.\right.\nonumber\\
&\left.\left.\cdot\left(D^2f(\mathbf{Y}_n)-D^2f(\mathbf{Y}_n^{ij})\right)\left[(\lfloor n\cdot \rfloor -2)\mathbbm{1}_{[i/n,1]\cap[j/n,1]},(\lfloor n\cdot \rfloor -2)\mathbbm{1}_{[i/n,1]\cap[j/n,1]}\right]\vphantom{\frac{1}{n^2}}\right\rbrace\vphantom{\sum_1^2}\right|\nonumber\\
\leq&\frac{\|g\|_{M}}{12n^2}\sum_{1\leq i\neq j\leq n}\mathbbm{E}\left[\left|(I_{ij}-2pI_{ij}+p)\right|\left\|\mathbf{Y}_n-\mathbf{Y}_n^{ij}\right\|\right]\label{eq_star},
\end{align}
where (\ref{eq_star}) follows from Proposition \ref{prop12.7}. Now,
$$\left\|\mathbf{Y}_n-\mathbf{Y}_n^{ij}\right\|\leq\frac{1}{n}\left|I_{ij}-I_{ij}'\right|$$
and so, by (\ref{eq_star}),
\begin{align}
\epsilon_2\leq&\frac{\|g\|_{M}}{12n^3}\sum_{1\leq i\neq j\leq n}\mathbbm{E}\left[\left|I_{i,j}-2pI_{i,j}+p\right|\left|I_{ij}-I_{ij}'\right|\right]
\leq\frac{\|g\|_{M}}{12n},\label{4_int}
\end{align}
where the last inequality holds because $|I_{ij}-2pI_{ij}+p|\leq 1$ and $|I_{ij}-I_{ij}'|\leq 1$ for all $i,j\in\lbrace 1,\cdots,n\rbrace$.

Using Theorem \ref{theorem1} together with (\ref{4_int}) and (\ref{4.1.1.1}) gives the desired result.
\end{proof}

\subsection{Distance from the continuous process}
We now establish a bound on the rate of convergence of $\mathbf{Y}_n$ to a continuous Gaussian process whose covariance is the limit of the covariance of $\mathbf{D}_n$. We do this by bounding the distance between $\mathbf{D}_n$ and the continuous process via the Brownian modulus of continuity and using Theorem \ref{theorem_pre_limiting}.
 \begin{theorem}\label{theorem_continuous}
Let $\mathbf{Y}_n$ be defined as in Subsection \ref{section1} and let $\mathbf{Z}$ be defined by:
$$
\mathbf{Z}(t)=\left(\frac{p(1-p)}{2}\right)^{1/2}t\mathbf{B}(t^2),\quad t\in[0,1],$$
where $\mathbf{B}$ is a standard Brownian Motion. Then, for any $g\in M^2$:
$$\left|\mathbbm{E}g(\mathbf{Y}_n)-\mathbbm{E}g(\mathbf{Z})\right|\leq \|g\|_{M^2}\left(9n^{-1}+90n^{-1/2}\sqrt{\log n}\right).$$
\end{theorem}
\begin{remark}
Theorem \ref{theorem_continuous}, together with Proposition \ref{prop_m}, implies that $\mathbf{Y}_n$ converges to $\mathbf{Z}$ in distribution with respect to the Skorokhod and uniform topologies.
\end{remark}
In \textbf{Step 1} of the proof of Theorem \ref{theorem_continuous}, we provide a coupling between $\mathbf{D}_n$ and a standard Brownian motion. Using this Brownian motion, we construct a process $\mathbf{Z}_n$ having the same distribution as $\mathbf{D}_n$. In \textbf{Step 2} we couple $\mathbf{Z}_n$ and $\mathbf{Z}$ and bound the first two moments of the supremum distance between them, using the Brownian modulus of continuity. In \textbf{Step 3} we use those bounds together with the Mean Value Theorem to obtain Theorem \ref{theorem_continuous}.
\begin{proof}[Proof of Theorem \ref{theorem_continuous}]~\\
\textbf{Step 1.}
Let $\mathbf{B}$ be a standard Brownian motion and let $\mathbf{Z}_n$ be defined by:
\begin{align*}
\mathbf{Z}_n(t)=&\frac{(\lfloor nt\rfloor -2)\sqrt{p(1-p)}}{n\sqrt{2}}\mathbf{B}\left(\frac{\lfloor nt\rfloor(\lfloor nt\rfloor -1)}{n^2}\right).
\end{align*}
Now, note that $\mathbf{D}_n\stackrel{\mathcal{D}}=\mathbf{Z}_n$. To see this, observe that for all $u,t\in[1/n,1]$,
\begin{align}
\mathbbm{E}\mathbf{D}_n(t)\mathbf{D}_n(u)=&(\lfloor nt\rfloor -2)(\lfloor nu\rfloor -2)\lfloor n(t\wedge u)\rfloor(\lfloor n(t\wedge u)\rfloor-1)\frac{p(1-p)}{2n^4}\nonumber\\
=&\mathbbm{E}\mathbf{Z}_n(t)\mathbf{Z}_n(u).\label{cov_structure}
\end{align}

\textbf{Step 2.} We let $\mathbf{Z}$ and $\mathbf{Z}_n$ be coupled in such a way that $\mathbf{Z}$ is constructed as in Theorem \ref{theorem_continuous}, using the same Brownian Motion $\mathbf{B}$, as the one used in the construction of $\mathbf{Z}_n$. In Lemma \ref{lemma10_app}, proved in the appendix, we derive bounds for moments of the supremum distance between $\mathbf{Z}$ and $\mathbf{Z}_n$:
\begin{align}
&\mathbbm{E}\left\|\mathbf{Z}_n-\mathbf{Z}\right\|\leq\frac{3}{\sqrt{2}n}+\frac{36\sqrt{\log n}}{\sqrt{n}}\nonumber\\
&\mathbbm{E}\left\|\mathbf{Z}_n-\mathbf{Z}\right\|^2\leq\frac{5}{n^2}+\frac{49\log n}{n}\nonumber\\
&\mathbbm{E}\|\mathbf{Z}\|^2\leq \frac{1}{2}.\label{4.1.2.1}
\end{align}

\textbf{Step 3.} We note that $\|Dg(w)\|\leq \|g\|_{M^{2}}(1+\|w\|)$ and therefore, by (\ref{4.1.2.1}):
\begin{align*}
\left|\mathbbm{E}g(\mathbf{Z})-\mathbbm{E}g(\mathbf{D}_n)\right|\stackrel{\text{MVT}}\leq&\mathbbm{E}\left[\sup_{c\in[0,1]}\left\|Dg(\mathbf{Z}+c(\mathbf{Z}_n-\mathbf{Z}))\right\|\|\mathbf{Z}-\mathbf{Z}_n\|\right]\\
\leq&\|g\|_{M^2}\mathbbm{E}\left[\sup_{c\in[0,1]}\left(1+\|\mathbf{Z}+c(\mathbf{Z}_n-\mathbf{Z})\|\right)\|\mathbf{Z}-\mathbf{Z}_n\|\right]\\
\leq&\|g\|_{M^{2}}\mathbbm{E}\left[\|\mathbf{Z}-\mathbf{Z}_n\|+\|\mathbf{Z}\|\|\mathbf{Z}-\mathbf{Z}_n\|+\|\mathbf{Z}-\mathbf{Z}_n\|^2\right]\\
\leq&\|g\|_{M^{2}}\left[\mathbbm{E}\|\mathbf{Z}-\mathbf{Z}_n\|+\sqrt{\mathbbm{E}\|\mathbf{Z}\|^2\mathbbm{E}\|\mathbf{Z}-\mathbf{Z}_n\|^2}+\mathbbm{E}\|\mathbf{Z}-\mathbf{Z}_n\|^2\right]\\
\leq&\|g\|_{M^2}\left(\left(\frac{3+\sqrt{5}}{\sqrt{2}}+5\right)n^{-1}+90n^{-1/2}\sqrt{\log n}\right),
\end{align*}
which, together with Theorem \ref{theorem_pre_limiting} gives the desired result.
\end{proof}
\begin{remark}
The representation of $\mathbf{Z}$ in terms of a Brownian motion comes from a careful analysis of the limiting covariance of $\mathbf{D}_n$. Indeed, (\ref{cov_structure}) provides an explicit derivation of the covariance, which converges to the covariance of $\mathbf{Z}$.
\end{remark}

\section*{Acknowledgements}
The author would like to thank Gesine Reinert for many helpful discussions and comments on the early versions of this work. The author is also grateful to Andrew Barbour, Christian D{\"o}bler, Ivan Nourdin, Giovanni Peccati and Yvik Swan for many interesting discussions. This research has been supported by an EPSRC PhD Studentship at the University of Oxford and the FNR grant FoRGES (R-AGR- 3376-10) at the University of Luxembourg.

\bibliographystyle{plain}
\bibliography{Bibliography}
\section{Appendix - technical details of the proofs of Theorems \ref{combi_pre_lim}, and \ref{theorem_continuous}}
\subsection{Technical details of the proof of Theorem \ref{combi_pre_lim}}
\begin{lemma}\label{lemma1_app}
In the setup of Theorem \ref{theorem_pre_limiting} and for $\epsilon_2$ defined by Theorem \ref{theorem1},
\begin{align*}
\epsilon_2=&\left|\frac{1}{2\lambda_n}\mathbbm{E}D^2f(\mathbf{Y}_n)\left[\mathbf{Y}_n-\mathbf{Y}_n',\mathbf{Y}_n-\mathbf{Y}_n'\right]-\mathbbm{E}D^2f(\mathbf{Y}_n)[\mathbf{D}_n,\mathbf{D}_n]\right|\nonumber\\
\leq& A+B,
\end{align*}
for
\begin{align}
A=&\left|\frac{1}{n(n-1)s_n^2}\underset{i\neq j,k\neq l}{\sum_{1\leq i,j,k,l\leq n}}\mathbbm{E}\left\lbrace\left[\frac{(X_{ik}-X_{il})^2}{2n}-\frac{\hat{Z}_i^2}{n-1}\right]\right.\left(D^2f(\mathbf{Y}_{n,ijkl})-D^2f\left(\mathbf{Y}_{n}^{ijkl}\right)\right)\left[\mathbbm{1}_{[i/n,1]}\mathbbm{1}_{[i/n,1]}\right]\vphantom{\left[\frac{1^2}{2^2}\right]}\right\rbrace\nonumber\\
&+\frac{1}{n(n-1)s_n^2}\underset{i\neq j,k\neq l}{\sum_{1\leq i,j,k,l\leq n}}\mathbbm{E}\left\lbrace\left[\frac{(X_{ik}-X_{il})(X_{jl}-X_{jk})}{2n}-\hat{Z}_i\hat{Z}_j\right]\right.\nonumber\\
&\left.\left.\hphantom{\frac{1}{n(n-1)s_n^2}\underset{i\neq j,k\neq l}{\sum_{1\leq i,j,k,l\leq n}}\mathbbm{E}}\cdot\left(D^2f(\mathbf{Y}_{n,ijkl})-D^2f\left(\mathbf{Y}_{n}^{ijkl}\right)\right)\left[\mathbbm{1}_{[i/n,1]},\mathbbm{1}_{[j/n,1]}\right]\vphantom{\left[\frac{1^2}{2^2}\right]}\right\rbrace\vphantom{\frac{1}{2}\underset{2}{\sum_{2}}}\right|\nonumber,\\
B=&\left|\frac{1}{n(n-1)s_n^2}\underset{i\neq j,k\neq l}{\sum_{1\leq i,j,k,l\leq n}}\mathbbm{E}\left\lbrace\left[\frac{(X_{ik}-X_{il})^2}{2n}-\frac{\hat{Z}_i^2}{n-1}\right]D^2f\left(\mathbf{Y}_{n}^{ijkl}\right)\left[\mathbbm{1}_{[i/n,1]}\mathbbm{1}_{[i/n,1]}\right]\right\rbrace\right.\nonumber\\
&\left.+\frac{1}{n(n-1)s_n^2}\underset{i\neq j,k\neq l}{\sum_{1\leq i,j,k,l\leq n}}\mathbbm{E}\left\lbrace\left[\frac{(X_{ik}-X_{il})(X_{jl}-X_{jk})}{2n}-\hat{Z}_i\hat{Z}_j\right]D^2f\left(\mathbf{Y}_{n}^{ijkl}\right)\left[\mathbbm{1}_{[i/n,1]},\mathbbm{1}_{[j/n,1]}\right]\vphantom{\left[\frac{1^2}{2^2}\right]}\right\rbrace\vphantom{\frac{1^2}{2}\underset{2}{\sum_{2}}}\right|.\nonumber
\end{align}
\end{lemma}
\begin{proof}
Note that
\begin{align}
\epsilon_2=&\left|\frac{1}{2\lambda_n}\mathbbm{E}D^2f(\mathbf{Y}_n)\left[(\mathbf{Y}_n-\mathbf{Y}_n'),\mathbf{Y}_n-\mathbf{Y}_n'\right]-\mathbbm{E}D^2f(\mathbf{Y}_n)[\mathbf{D}_n,\mathbf{D}_n]\right|\nonumber\\
=&\left|\frac{n-1}{4}\mathbbm{E}D^2f(\mathbf{Y}_n)[\mathbf{Y}_n-\mathbf{Y}_n',\mathbf{Y}_n-\mathbf{Y}_n']-\mathbbm{E}D^2f(\mathbf{Y}_n)[\mathbf{D}_n,\mathbf{D}_n]\right|\label{eps222}
\end{align}
and
\begin{align}
&\frac{n-1}{4}\mathbbm{E}D^2f(\mathbf{Y}_n)[\mathbf{Y}_n-\mathbf{Y}_n',\mathbf{Y}_n-\mathbf{Y}_n']-\mathbbm{E}D^2f(\mathbf{Y}_n)[\mathbf{D}_n,\mathbf{D}_n]\nonumber\\
=&\frac{1}{2ns_n^2}\sum_{i,j=1}^n\mathbbm{E}\left\lbrace(X_{i\pi(i)}-X_{i\pi(j)})^2D^2f(\mathbf{Y}_n)\left[\mathbbm{1}_{[i/n,1]}\mathbbm{1}_{[i/n,1]}\right]\right\rbrace\nonumber\\
&+\frac{1}{2ns_n^2}\sum_{i,j=1}^n\mathbbm{E}\left\lbrace(X_{i\pi(i)}-X_{i\pi(j)})(X_{j\pi(j)}-X_{j\pi(i)})D^2f(\mathbf{Y}_n)\left[\mathbbm{1}_{[i/n,1]},\mathbbm{1}_{[j/n,1]}\right]\right\rbrace -\mathbbm{E}D^2f(\mathbf{Y}_n)[\mathbf{D}_n,\mathbf{D}_n]\nonumber\\
=&\frac{1}{2n^2(n-1)s_n^2}\underset{i\neq j,k\neq l}{\sum_{1\leq i,j,k,l\leq n}}\mathbbm{E}\left\lbrace(X_{ik}-X_{il})^2\left.\cdot D^2f(\mathbf{Y}_n)\left[\mathbbm{1}_{[i/n,1]}\mathbbm{1}_{[i/n,1]}\right]\right|\pi(i)=k,\pi(j)=l\right\rbrace\nonumber\\
&+\frac{1}{n(n-1)s_n^2}\underset{i\neq j,k\neq l}{\sum_{1\leq i,j,k,l\leq n}}\mathbbm{E}\left\lbrace\left[\frac{(X_{ik}-X_{il})(X_{jl}-X_{jk})}{2n}\right]D^2f(\mathbf{Y}_n)\left[\mathbbm{1}_{[i/n,1]},\mathbbm{1}_{[j/n,1]}\right]\left.\vphantom{\frac{1}{2}}\right|\pi(i)=k,\pi(j)=l\vphantom{\frac{1}{2}}\right\rbrace\nonumber\\
&-\frac{1}{s_n^2}\sum_{1\leq i\neq j\leq n}\mathbbm{E}[\hat{Z}_i\hat{Z}_j]\mathbbm{E}D^2f(\mathbf{Y}_n)[\mathbbm{1}_{[i/n,1]},\mathbbm{1}_{[j/n,1]}]-\frac{1}{(n-1)s_n^2}\sum_{1\leq i\neq j\leq n}\mathbbm{E}[\hat{Z}_i^2]\mathbbm{E}D^2f(\mathbf{Y}_n)[\mathbbm{1}_{[i/n,1]},\mathbbm{1}_{[i/n,1]}]\nonumber\\
=&\frac{1}{2n^2(n-1)s_n^2}\underset{i\neq j,k\neq l}{\sum_{1\leq i,j,k,l\leq n}}\mathbbm{E}\left\lbrace (X_{ik}-X_{il})^2 D^2f(\mathbf{Y}_{n,ijkl})\left[\mathbbm{1}_{[i/n,1]}\mathbbm{1}_{[i/n,1]}\right]\right\rbrace\nonumber\\
&+\frac{1}{n(n-1)s_n^2}\underset{i\neq j,k\neq l}{\sum_{1\leq i,j,k,l\leq n}}\mathbbm{E}\left\lbrace\frac{(X_{ik}-X_{il})(X_{jl}-X_{jk})}{2n}D^2f(\mathbf{Y}_{n,ijkl})\left[\mathbbm{1}_{[i/n,1]},\mathbbm{1}_{[j/n,1]}\right]\right\rbrace\nonumber\\
&-\frac{1}{n(n-1)s_n^2}\underset{i\neq j,k\neq l}{\sum_{1\leq i,j,k,l\leq n}}\mathbbm{E}[\hat{Z}_i\hat{Z}_j]\mathbbm{E}D^2f(\mathbf{Y}_{n,ijkl})[\mathbbm{1}_{[i/n,1]},\mathbbm{1}_{[j/n,1]}]\nonumber\\
&-\frac{1}{n(n-1)^2s_n^2}\underset{i\neq j,k\neq l}{\sum_{1\leq i,j,k,l\leq n}}\mathbbm{E}[\hat{Z}_i^2]\mathbbm{E}D^2f(\mathbf{Y}_{n,ijkl})[\mathbbm{1}_{[i/n,1]},\mathbbm{1}_{[i/n,1]}].\label{eps22}
\end{align}
Now, the lemma follows by taking the absolute value in (\ref{eps22}) and combining it with (\ref{eps222}).
\end{proof}
\begin{lemma}\label{lemma2_app}
For $A$ of (\ref{a_def}),
\begin{align*}
A\leq&\frac{\|g\|_{M^1}}{n^3(n-1)s_n^3}\sum_{1\leq i,j,k,l,u\leq n}\left\lbrace\vphantom{\sum_1^2}\mathbbm{E}|X_{ik}|^3+5\mathbbm{E}|X_{ik}|\mathbbm{E}|X_{il}|^2+7\mathbbm{E}|X_{ik}|^2\mathbbm{E}|X_{jl}|\right.\nonumber\\
&+5\mathbbm{E}|X_{ik}|^2\mathbbm{E}|X_{jk}|+16\mathbbm{E}|X_{ik}|\mathbbm{E}|X_{il}|\mathbbm{E}|X_{jl}|+2\mathbbm{E}|X_{iu}|\mathbbm{E}|X_{ik}|\mathbbm{E}|X_{il}|\nonumber\\
&+4\mathbbm{E}|X_{iu}|\mathbbm{E}|X_{il}|\mathbbm{E}|X_{jk}|+6\mathbbm{E}|X_{uk}|\mathbbm{E}|X_{ik}|\mathbbm{E}|X_{jl}|+2\mathbbm{E}|X_{uk}|\mathbbm{E}|X_{ik}|\mathbbm{E}|X_{jk}|\nonumber\\
&\left.+\frac{1}{n}\left(2\mathbbm{E}\left|X_{ik}\right|+2\mathbbm{E}\left|X_{jl}\right|+2\mathbbm{E}|X_{uk}|+2\mathbbm{E}|X_{ul}|\right)\cdot\sum_{r=1}^n\left(\mathbbm{E}|X_{ir}|^2+|c_{ir}c_{jr}|\right)\right\rbrace.\nonumber
\end{align*}
\end{lemma}
\begin{proof}
Let us adopt the notation of the proof of Theorem \ref{combi_pre_lim}. Define index sets $\mathcal{I}=\lbrace i,j,\pi^{-1}(k),\pi^{-1}(l)\rbrace$ and $\mathcal{J}=\lbrace k,l,\pi(i),\pi(j)\rbrace$. Then, letting $\mathbf{S}=\frac{1}{s_n}\sum_{i'\not\in\mathcal{I}}X_{i'\pi(i')}\mathbbm{1}_{[i'/n,1]}$, we can write:
$$\mathbf{Y}_{n,ijkl}=\mathbf{S}+\frac{1}{s_n}\sum_{i'\in\mathcal{I}}X_{i'\pi_{ijkl}(i')}\mathbbm{1}_{[i'/n,1]},\qquad \mathbf{Y}_n^{ijkl}=\mathbf{S}+\frac{1}{s_n}\sum_{i'\in\mathcal{I}}X_{i'\pi(i')}^{ijkl}\mathbbm{1}_{[i'/n,1]}.$$
Since $\mathbf{S}$ depends only on the components of $\mathbbm{X}$ outside the square $\mathcal{I}\times\mathcal{J}$ and $\lbrace \pi(i):i\not\in\mathcal{I}\rbrace$, $\mathbf{S}$ is independent of:
$$\left\lbrace X_{il},X_{jk},X_{ik},X_{jl},\sum_{i'\in\mathcal{I}}X_{i'\pi_{ijkl}(i')},\sum_{i'\in\mathcal{I}}X_{i'\pi(i')}^{ijkl}\right\rbrace,$$
given $\pi^{-1}(k), \pi^{-1}(l),\pi(i),\pi(j)$.

Note that, by Proposition \ref{prop12.7},
\begin{align}
A\leq&\frac{\|g\|_{M^1}}{n(n-1)s_n^2}\underset{i\neq j,k\neq l}{\sum_{1\leq i,j,k,l\leq n}}\mathbbm{E}\left\lbrace\left\|\mathbf{Y}_{n,ijkl}-\mathbf{Y}_n^{ijkl}\right\|\left(\left|\frac{(X_{ik}-X_{il})^2}{2n}-\frac{\mathbbm{E}\hat{Z}_i^2}{n-1}\right|\right.\right.\nonumber\\
&\hphantom{\frac{\|g\|_{M^1}}{n(n-1)s_n^2}\underset{i\neq j,k\neq l}{\sum_{1\leq i,j,k,l\leq n}}\mathbbm{E}\leq}\left.\left.+\left|\frac{(X_{ik}-X_{il})(X_{jl}-X_{jk})}{2n}-\mathbbm{E}(\hat{Z}_i\hat{Z}_j)\right|\right)\right\rbrace\nonumber\\
\leq&\frac{\|g\|_{M^1}}{n(n-1)s_n^3}\underset{i\neq j,k\neq l}{\sum_{1\leq i,j,k,l\leq n}}\mathbbm{E}\left\lbrace\sum_{i'\in\mathcal{I}}\vphantom{\left|\frac{1^2}{1^2}\right|}\left|X_{i',\pi_{ijkl}(i')}-X_{i'\pi(i')}^{ijkl}\right|\right.\nonumber\\
&\left.\cdot\left(\left|\frac{(X_{ik}-X_{il})^2}{2n}-\frac{\mathbbm{E}\hat{Z}_i^2}{n-1}\right|+\left|\frac{(X_{ik}-X_{il})(X_{jl}-X_{jk})}{2n}-\mathbbm{E}(\hat{Z}_i\hat{Z}_j)\right|\right)\right\rbrace\nonumber\\
\leq& \frac{\|g\|_{M^1}}{n(n-1)s_n^3}\underset{i\neq j,k\neq l}{\sum_{1\leq i,j,k,l\leq n}}\mathbbm{E}\left\lbrace\vphantom{\frac{1^2}{2^2}}\left(\left|X_{ik}-X^{ijkl}_{i\pi(i)}\right|+\left|X_{jl}-X^{ijkl}_{j\pi(j)}\right|+\left|X_{\pi^{-1}(k),k}-X_{\pi^{-1}(k),k}^{ijkl}\right|+\left|X_{\pi^{-1}(l),l}-X_{\pi^{-1}(l),l}^{ijkl}\right|\right)\right.\nonumber\\
&\left.\cdot\left(\left|\frac{(X_{ik}-X_{il})^2}{2n}-\frac{\mathbbm{E}\hat{Z}_i^2}{n-1}\right|+\left|\frac{(X_{ik}-X_{il})(X_{jl}-X_{jk})}{2n}-\mathbbm{E}(\hat{Z}_i\hat{Z}_j)\right|\right)\right\rbrace\nonumber\\
\leq&\frac{\|g\|_{M^1}}{2n(n-1)^2s_n^3}\underset{i\neq j,k\neq l}{\sum_{1\leq i,j,k,l\leq n}}\mathbbm{E}\left\lbrace\left(\left|X_{ik}\right|+\left|X^{ijkl}_{i,\pi(i)}\right|+\left|X_{j,l}\right|+\left|X^{ijkl}_{j,\pi(j)}\right|+\left|X_{\pi^{-1}(k),k}\right|+\left|X_{\pi^{-1}(k),k}^{ijkl}\right|\right.\right.\nonumber\\
&\left.+\left|X_{\pi^{-1}(l),l}\right|+\left|X_{\pi^{-1}(l),l}^{ijkl}\right|\right)\left(\left|X_{ik}\right|^2+\left|X_{il}\right|^2+2\left|X_{ik}X_{il}\right|+2\mathbbm{E}\left|\hat{Z}_i\right|^2+\left|X_{ik}X_{jl}\right|+\left|X_{ik}X_{jk}\right|\right.\nonumber\\
&\left.\left.+\left|X_{il}X_{jl}\right|+\left|X_{il}X_{jk}\right|+2(n-1)\left|\mathbbm{E}(\hat{Z}_i\hat{Z}_j)\right|\right)\right\rbrace\nonumber\\
\leq&\frac{\|g\|_{M^1}}{2n^2(n-1)^2s_n^3}\underset{i\neq j,k\neq l}{\sum_{1\leq i,j,k,l,u\leq n}}\mathbbm{E}\left\lbrace\vphantom{\sum_i^j}\left(\left|X_{ik}\right|+\mathbbm{E}\left|X_{iu}\right|+\left|X_{j,l}\right|+\mathbbm{E}\left|X_{ju}\right|+\left|X_{uk}\right|+\mathbbm{E}\left|X_{uk}\right|+\left|X_{ul}\right|+\mathbbm{E}\left|X_{ul}\right|\right)\right.\nonumber\\
&\cdot\left(\left|X_{ik}\right|^2+\left|X_{il}\right|^2+2\left|X_{ik}X_{il}\right|+\frac{2}{n}\sum_{r=1}^n\mathbbm{E}|X_{ir}|^2+\left|X_{ik}X_{jl}\right|+\left|X_{ik}X_{jk}\right|\right.\nonumber\\
&\left.\left.\hphantom{.......}+\left|X_{il}X_{jl}\right|+\left|X_{il}X_{jk}\right|+\frac{2}{n}\left|\sum_{r=1}^nc_{ir}c_{jr}\right|\right)\right\rbrace\nonumber\\
\leq&\frac{\|g\|_{M^1}}{2n^2(n-1)^2s_n^3}\sum_{1\leq i,j,k,l,u\leq n}\left\lbrace\vphantom{\sum_i^j} \mathbbm{E}|X_{ik}|^3+5\mathbbm{E}|X_{ik}|^2\mathbbm{E}|X_{il}|+9\mathbbm{E}|X_{ik}|^2\mathbbm{E}|X_{jl}|+5\mathbbm{E}|X_{ik}|^2\mathbbm{E}|X_{jk}|\right.\nonumber\\
&+14\mathbbm{E}|X_{ik}|\mathbbm{E}|X_{il}|\mathbbm{E}|X_{jl}| +2\mathbbm{E}|X_{ik}|\mathbbm{E}|X_{il}|\mathbbm{E}|X_{iu}|+6\mathbbm{E}|X_{ik}|\mathbbm{E}|X_{jl}|\mathbbm{E}|X_{iu}|+2\mathbbm{E}[|X_{uk}||X_{ik}|^2]\nonumber\\
&+4\mathbbm{E}|X_{uk}X_{ik}|\mathbbm{E}|X_{il}|+4\mathbbm{E}|X_{uk}X_{ik}|\mathbbm{E}|X_{jl}|+8\mathbbm{E}|X_{uk}|\mathbbm{E}|X_{ik}|\mathbbm{E}|X_{jl}|+2\mathbbm{E}|X_{uk}X_{ik}X_{jk}|\nonumber\\
&\left.+2\mathbbm{E}|X_{uk}|\mathbbm{E}|X_{ik}|\mathbbm{E}|X_{jk}|+\frac{2}{n}\mathbbm{E}\left(2\left|X_{ik}\right|+6\left|X_{jl}\right|\right)\cdot\left(\sum_{r=1}^n\mathbbm{E}|X_{ir}|^2+\left|\sum_{r=1}^nc_{ir}c_{jr}\right|\right)\right\rbrace.\nonumber
\end{align}
which finishes the proof.
\end{proof}

\subsection{Technical details of the proof of Theorem \ref{theorem_continuous}}
\begin{lemma}\label{lemma10_app}
Using the notation of \textbf{Step 2} of the proof of Theorem \ref{theorem_continuous},
\begin{align*}
&\mathbbm{E}\left\|\mathbf{Z}_n-\mathbf{Z}\right\|\leq\frac{3}{\sqrt{2}n}+\frac{36\sqrt{\log n}}{\sqrt{n}}\\
&\mathbbm{E}\left\|\mathbf{Z}_n-\mathbf{Z}\right\|^2\leq\frac{5}{n^2}+\frac{49\log n}{n}\\
&\mathbbm{E}\|\mathbf{Z}\|^2\leq \frac{1}{2}.
\end{align*}
\end{lemma}
\begin{proof}
Note the following
\begin{enumerate}
\item By Doob's $L^2$ inequality, 
\begin{equation}\label{fourth_in_1}
\mathbbm{E}\left[\sup_{t\in[0,1]}|\mathbf{B}(t^2)|\right]\leq 2\quad\text{and}\quad\left|\frac{\lfloor nt\rfloor-2}{n}-t\right|\leq \frac{3}{n}, \quad\text{ for all }t\in[0,1].
\end{equation}
\item Using  \cite[Lemma 3]{ito_processes} and the fact that 
$$\left|\frac{\lfloor nt\rfloor(\lfloor nt\rfloor-1)}{n^2}-t^2\right|\leq \left|\frac{(nt-\lfloor nt\rfloor)(nt+\lfloor nt\rfloor)}{n^2}\right|+\frac{1}{n}\leq \frac{3}{n},$$
we obtain
\begin{equation}\label{fourth_in_2}
\mathbbm{E}\left[\sup_{t\in[0,1]}\left|\mathbf{B}\left(\frac{\lfloor nt\rfloor(\lfloor nt\rfloor -1)}{n^2}\right)-\mathbf{B}(t^2)\right|\right]\leq \frac{30\sqrt{3\log\left(\frac{2n}{3}\right)}}{n^{1/2}\sqrt{\pi\log(2)}}.
\end{equation}
\end{enumerate}

Now, we can bound $\mathbbm{E}\left\|\mathbf{Z}_n-\mathbf{Z}\right\|$ in the following way:
\begin{align*}
\mathbbm{E}\left\|\mathbf{Z}_n-\mathbf{Z}\right\|
\leq&\frac{\sqrt{p(1-p)}}{\sqrt{2}}\mathbbm{E}\left[\sup_{t\in[0,1]}\left|\frac{\lfloor nt\rfloor -2}{n}\mathbf{B}\left(\frac{\lfloor nt\rfloor(\lfloor nt\rfloor -1)}{n^2}\right)-t\mathbf{B}(t^2)\right|\right]\nonumber\\
\leq&\frac{\sqrt{p(1-p)}}{\sqrt{2}}\left\lbrace\mathbbm{E}\left[\sup_{t\in[0,1]}\left|\left(\frac{\lfloor nt\rfloor -2}{n}-t\right)\mathbf{B}(t^2)\right|\right]\right.\nonumber\\
&\left.\hphantom{\frac{\sqrt{p(1-p)}}{\sqrt{2}}}+\mathbbm{E}\left[\sup_{t\in[0,1]}\frac{\lfloor nt\rfloor -2}{n}\left|\mathbf{B}\left(\frac{\lfloor nt\rfloor(\lfloor nt\rfloor -1)}{n^2}\right)-\mathbf{B}(t^2)\right|\right]\right\rbrace\nonumber\\
\stackrel{(\ref{fourth_in_1}),(\ref{fourth_in_2})}\leq&\frac{\sqrt{p(1-p)}}{\sqrt{2}}\left(\frac{6}{n}+\frac{30\sqrt{3\log n}}{n^{1/2}\sqrt{\pi\log(2)}}\right)\nonumber\\
\leq&\frac{3}{\sqrt{2}n}+\frac{36\sqrt{\log n}}{\sqrt{n}}.
\end{align*}

Similarly, using Doob's $L^2$ inequality and \cite[Lemma 3]{ito_processes},
\begin{align*}
\mathbbm{E}\|\mathbf{Z}_n-\mathbf{Z}\|^2
\leq&\frac{p(1-p)}{2}\mathbbm{E}\left[\sup_{t\in[0,1]}\left|\frac{\lfloor nt\rfloor -2}{n}\mathbf{B}\left(\frac{\lfloor nt\rfloor(\lfloor nt\rfloor -1)}{n^2}\right)-t\mathbf{B}(t^2)\right|^2\right]\nonumber\\
\leq& \frac{p(1-p)}{2}\left(\mathbbm{E}\left[\sup_{t\in[0,1]}\left|\left(\frac{\lfloor nt\rfloor -2}{n}-t\right)\mathbf{B}(t^2)\right|^2\right]+\mathbbm{E}\left[\sup_{t\in[0,1]}\left|\mathbf{B}\left(\frac{\lfloor nt\rfloor(\lfloor nt\rfloor -1)}{n^2}\right)-\mathbf{B}(t^2)\right|^2\right]\right)\\
\leq &\frac{18p(1-p)}{n^2}+ \frac{135p(1-p)}{\log(2)}\frac{\log\left(\frac{2}{3}n\right)}{n}
\nonumber.\\
\leq& \frac{5}{n^2}+\frac{49\log n}{n}.
\end{align*}
Furthermore, by Doob's $L^2$ inequality,
\begin{align*}
\mathbbm{E}\|\mathbf{Z}\|^2\leq&\frac{p(1-p)}{2}\mathbbm{E}\left[\sup_{t\in[0,1]}t^2\left|\mathbf{B}_1(t^2)\right|^2\right]\leq 2p(1-p)\leq\frac{1}{2}.
\end{align*}
This finishes the proof.
\end{proof}
\end{document}